%% file: hist_main.tex
\documentclass[preprint]{imsart}

\RequirePackage[OT1]{fontenc}
\RequirePackage{amsthm,amsmath,amssymb}
\RequirePackage[numbers]{natbib}
\RequirePackage[colorlinks,citecolor=blue,urlcolor=blue]{hyperref}
\RequirePackage{graphicx}
\usepackage{subcaption}

\startlocaldefs
\input ing_macros.tex

\endlocaldefs

\begin{document}

\begin{frontmatter}

\title{Improved Classification Rates under Refined Margin Conditions}


\begin{aug}
\author{\fnms{Ingrid} \snm{Blaschzyk}\ead[label=e1]{ingrid.blaschzyk@mathematik.uni-stuttgart.de, ingo.steinwart@mathematik.uni-stuttgart.de}}
\and
\author{\fnms{Ingo} \snm{Steinwart}\ead[label=e2]{ingo.steinwart@mathematik.uni-stuttgart.de}}
\address{Institute for Stochastics and Applications\\University of Stuttgart\\Pfaffenwaldring 57\\ D-70569 Stuttgart \\ \printead{e1}}
\runauthor{I. Blaschzyk and I. Steinwart}
\runtitle{Classification Rates under Refined Margin Conditions}
\end{aug}

\begin{abstract}
In this paper we present a simple partitioning based technique to refine the statistical analysis of classification algorithms. The core of our idea is to divide the input space into two parts such that the first part contains a suitable vicinity around the decision boundary, while the second part is sufficiently far away from the decision boundary. Using a set of margin conditions we are then able to control the classification error on both parts separately. By balancing out these two error terms we obtain a refined error analysis in a final step. We apply this general idea to the histogram rule and show that even for this simple method we obtain, under certain assumptions, better rates than the ones known for support vector machines, for certain plug-in classifiers, and for a recently analyzed tree based adaptive-partitioning ansatz. Moreover, we show that a  margin condition which sets the critical noise in relation to the
decision boundary makes it possible to improve the optimal
rates proven for distributions without this margin condition.
\end{abstract}
%
%


\tableofcontents

\end{frontmatter}

\section{Introduction}
Given a dataset $D:=((x_1,y_1),\ldots,(x_n,y_n))$ of observations drawn in an i.i.d.\ fashion from a probability measure $P$ on $X \times Y$, where $X \subset \mathbb{R}^d$ and $Y:=\lbrace -1,1 \rbrace$, the learning goal of binary classification is to find a decision function $f_D \colon X \to  \lbrace{-1,1\rbrace}$ such that for new data $(x,y)$ we have $f_D(x)=y$ with high probability.

The problem of classification is, apart from regression, one of the most considered problems in learning theory and many classical learning methods have been presented in the literature such as histogram rules, nearest neighbor methods or moving window rules. A general reference for these methods is \cite{DeGyLu96}. 
Several more recent methods use trees to build a classifier, for example the random forest algorithm, introduced in \cite{Breiman01}, makes a prediction by a majority vote over a collection of random forest trees. 
Another example is the tree based adaptive-partitioning algorithm, presented in \cite{BiCoDaDe14}. Here, a classifier is picked by empirical risk minimization over a nested sequence $(S_m)_{m \geq 1}$ of families of sets which is based on dyadic or decorated tree partitions. Examples of non-tree based algorithms are described in \citep{AuTs07} and \citep{KoKr07}. There, the final classifier is found by empirical risk minimization over a suitable grid of plug-in rules or is derived by plug-in kernel, partitioning or nearest neighbor classification rules. Another non-tree based algorithm is, for example, the support vector machine (SVM), which solves a regularized empirical risk minimization problem over a reproducing kernel Hilbert space $H$. For more details on statistical properties of SVM for classification we refer the reader to \cite[Chapter~8]{StCh08}. 

In this paper we discuss a partitioning based technique to analyse the statistical properties of classification algorithms. In particular we show for the histogram rule that under certain assumptions this technique leads to rates, which are faster than the rates obtained in \cite{AuTs07, BiCoDaDe14, KoKr07}, and \cite{StCh08}. To be more precise, we divide the input space $X$ into two overlapping regions that are adjustable by a parameter $r$ in such a way that one set, which we will denote by $N_r$, contains points near the decision boundary, whereas the other set $F_r$ contains those that are sufficiently far away from the decision boundary. We examine the excess risks over these two sets separately by applying an oracle inequality for empirical risk minimizers on both parts. It turns out that we have no approximation error on $F_r$ and that we obtain, under a suitable assumption which relates critical noise to the decision boundary, an optimal variance bound on $F_r$, which in turn leads to an $\mathcal{O}(n^{-1})$ behavior of the excess risk on $F_r$. However, this bound still depends on the parameter $r$, namely it increases for $r \rightarrow 0$. In contrast, our bound on the risk on $N_r$ decreases for $r \rightarrow 0$. By balancing out these two risks with respect to $r$ we obtain a refined bound on $X$ under additional assumptions describing the concentration of mass around the decision boundary.

A more detailed discussion on this technique and the statistical results, which include rate adaptivity, are presented in Section \ref{section_main}. Moreover, a comparison of the resulting learning rates to the ones known for the SVM, for certain plug-in classification rules and the tree based adaptive-partitioning algorithm described in \cite{BiCoDaDe14} can be found at the end of Section \ref{section_comparison}. In particular, we show that the above mentioned assumption that relates the location of critical noise to the decision boundary has an essential influence on our learning rates such that we outperform under a common set of assumptions the optimal rates obtained for the classifier in \citep{AuTs07}. Furthermore, we show that if we omit the latter assumption, we obtain exactly the optimal rate of \citep{AuTs07}. We note that all proofs are deferred to Section \ref{section_proofs}.

\section{General Assumptions}\label{section_assump}

To describe our learning goal we consider in the following the classification loss $L:=\lclass: Y \times \mathbb{R} \to [0,\infty)$, defined by $L(y,t):=\eins_{(-\infty,0] }(y \cdot \text{sign} t)$ for $y\in Y, t \in \R$, where $\eins_{(-\infty,0]}$ denotes the indicator function on $(-\infty,0]$. We define the risk of a measurable estimator $f: X \to \mathbb{R}$ by
\begin{align*}
\Rx{L}{P}{f}:=\int_{X \times Y} L(y,f(x)) \,dP(x,y)
\end{align*}
and the empirical risk by
\begin{align*}
\Rx{L}{D}{f}:=\frac{1}{n} \sum_{i=1}^n L(y_i,f(x_i)),
\end{align*}
where $D:=\frac{1}{n} \sum_{i=1}^n \d_{(x_i,y_i)}$ denotes the average of Dirac measures  $\d_{(x_i,y_i)}$ at $(x_i,y_i)$. The smallest possible risk 
\begin{align*}
\RPB{L}:=\inf_{f \colon X \to \mathbb{R}}\Rx{L}{P}{f}
\end{align*}
is called the Bayes risk, and a measurable function $f^{\ast}_{L,P} \colon X \to \mathbb{R}$ so that
$\Rx{L}{P}{f^{\ast}_{L,P}}=\RPB{L}$ holds is called Bayes decision function. Recall that the Bayes decision function $f^{\ast}_{L,P}$ for the classification loss is given by $\text{sign}(2P(y=1|x)-1)$ for $x \in X$, where $P(\,\cdot \,|x)$ is a regular conditional probability on $Y$ given $x$.\\
Let us now briefly describe a particular histogram rule. To this end, let $\mathcal{A}=(A_j)_{j \geq 1}$ be a partition of $\mathbb{R}^d$ into cubes of side length $s \in (0,1]$ and $X:=\left[-1,1 \right]^d$. For $x \in X$ we denote by $A(x)$ the unique cell of $\mathcal{A}$ with $x \in A(x)$ and call the map $h_{P,s} \colon X \to Y$ defined by
\begin{align}\label{def_hist}
h_{P,s}(x)  :=\begin{cases}
  -1  & \text{if}\, f_{P,s}(x)<0,\\
  1  & \text{if}\, f_{P,s}(x)\geq 0,
\end{cases}
\end{align}
where $f_{P,s}(x):=P(A(x)\times \lbrace 1 \rbrace )-P(A(x) \times \lbrace -1 \rbrace )$, infinite sample histogram rule. For a dataset $D$ we further write
\begin{align*}
f_{D,s}(x):=\frac{1}{n}\sum_{i=1}^n \eins_{\lbrace y_i=+1 \rbrace} \eins_{A(x)}(x_i)-\frac{1}{n}\sum_{i=1}^n \eins_{\lbrace y_i=-1 \rbrace} \eins_{A(x)}(x_i).
\end{align*}
Thus, the empirical histogram is defined by $h_{D,s}:=\text{sign}f_{D,s}$. We define the set $\F$ by
\begin{align*}
\F:=\left\lbrace \,\sum_{A_j \cap \left[-1,1 \right]^d \neq \emptyset}c_j\eins_{A_j} \colon c_j\in \lbrace -1,1 \rbrace  \, \right\rbrace.
\end{align*}
Then, it is easy to show that the empirical histogram rule $h_{D,s}$ is an empirical risk minimizer over $\F$ for the classification loss, that means
\begin{align*}
\Rx{L}{D}{h_{D,s}}=\inf_{f\in \F}\Rx{L}{D}{f}.
\end{align*}
Since we aim in a further step to examine the risk on subsets of $X$ consisting of cells, we have to specify the loss on those subsets. Therefore, we define for an arbitrary index set $J \subset \lbrace \,1,\ldots,m \,\rbrace$ the set 
\begin{align}\label{def_set_T}
T_J:= \bigcup_{j\in J} A_j
\end{align}
and the related loss $L_{T_J}: X \times Y \times \mathbb{R} \rightarrow [0, \infty)$ by
\begin{align}\label{loss_A}
L_{T_J}(x,y,t):=\eins_{\bigcup_{j\in J} A_j}(x)\lclass(y,t).
\end{align}
Furthermore, we define the risk over $T_J$ by
\begin{align*}
\Rx{L_{T_J}}{P}{f}:=\int_{X \times Y} L_{T_J}(x,y,f(x))  \,dP(x,y)
\end{align*}
and define the shortcut $L_{T_J} \circ f:=L_{T_J}(x,y,f(x))$.

We denote by $P^n$ the product measure of the probability measure $P$. As mentioned in the introduction, we have to make assumptions on $P$ to obtain rates. Therefore, we recall some notions from \cite[Chapter 8]{StCh08} which describe the behavior of $P$ in the vicinity of the decision boundary. To this end, let $\n \colon X \to [0,1]$, defined by $\n(x):=P(y=1|x)$ for $x\in X$, be a version of the posterior probability of $P$, that is, that the probability measures $P(\,\cdot \,|x)$ form a regular conditional probability of P. Clearly, if we have $\n(x)=0$ resp.\ $\n(x)=1$ for $x \in X$ we observe the label $y=-1$ resp. $y=1$ with probability $1$. Otherwise, if, e.g., $\n(x)\in [1/2,1)$  we observe the label $y=-1$ with the probability $1-\n(x) \in (0,1/2]$ and we call the latter probability noise. Obviously, in the worst case this probability equals $1/2$ and we define the set containing those $x \in X$ by $X_{0} := \lbrace\,x \in X \colon \n(x)=1/2 \,\rbrace$. Furthermore, we write
\begin{align*}
X_{1} &:= \lbrace\,x \in X \colon \n(x)>1/2 \,\rbrace,\\
X_{-1} &:=  \lbrace\, x \in X \colon \n(x)<1/2 \,\rbrace.
\end{align*}
Then, the function $\D_{\n} \colon X \to [0,\infty]$ defined by
\begin{align}\label{delta}
\begin{split}
\D_{\n}(x):=\begin{cases}
 d(x,X_1)  & \text{if}\, x \in X_{-1},\\
  d(x,X_{-1})  & \text{if}\, x \in X_{1},\\
  0 & \text{otherwise},
\end{cases}
\end{split}
\end{align}
where $d(x,A):=\inf_{x' \in A}d(x,x')$, is called distance to the decision boundary. This helps us to describe the mass of the marginal distribution $P_X$ of $P$ around the decision boundary by the following exponents. We say that $P$ has margin exponent (ME) $\a \in (0,\infty]$ if there exists a constant $c_{\text{ME}}>0$ such that
\begin{align*}
P_X(\lbrace \D_{\n}(x)<t\rbrace) \leq (c_{\text{ME}}t)^{\a}
\end{align*}
for all $t>0$. Descriptively, the ME $\a$ measures the amount of mass close to the decision boundary. Therefore, large values of $\a$ are better since they reflect a low concentration of mass in this region, which makes the classification easier. Furthermore, we say that $P$ has margin-noise exponent (MNE) $\b \in (0,\infty]$ if there exists a constant $c_{\text{MNE}} > 0$ such that
\begin{align*}
\int_{\lbrace \D_{\n}<t \rbrace} |2\n(x)-1| \,dP_X(x) \leq (c_{\text{MNE}}t)^{\beta}
\end{align*}
for all $t>0$. The MNE $\b$ measures the mass and the noise, that means the amount of points $x \in X$ with $\n(x)\approx 1/2$, around the decision boundary. That is, we have high MNE $\b$ if we have low mass and/or high noise around the decision boundary. Next, we say that the distance to the decision boundary $\D_{\n}$ controls the noise from below by the exponent $\g$ if there exist a $\g \in [0,\infty)$ and a constant $c_{\text{LC}}>0$ with 
\begin{align}\label{def_lc}
\D_{\n}^{\g}(x) \leq c_{\text{LC}}|2\n(x) -1|
\end{align}
for $P_X$-almost all $x \in X$. That means, if $\n(x)$ is close to $1/2$ for some $x \in X$, this $x$ is close to the decision boundary. For examples of typical values of these exponents and relations between them we refer the reader to \cite[Chapter~8]{StCh08}.

Finally, in order to describe the region of the decision boundary in a more geometrical way, we say according to \cite[3.2.14(1)]{Federer69} that a general set $T \subset X$ is $m$-rectifiable for an integer $m>0$ if there exists a Lipschitzian function mapping some bounded subset of $\mathbb{R}^m$ onto $T$. Furthermore, we denote by $\partial_X T$ the relative boundary of $T$ in $X$. Moreover, we denote by $\mathcal{H}^{d-1}$ the $(d-1)$-dimensional Hausdorff measure on $\mathbb{R}^d$, see \cite[Introduction]{Federer69}. 
The following lemma, which is based on \cite[Lemma A.10.4]{Steinwart15a}, describes the Lebesgue measure of the decision boundary in terms of the Hausdorff measure. Its result will be necessary for the analysis of the main theorem in Section \ref{section_main}.

\input{lemma_hausdorff}

\section{Oracle Inequality and Learning Rates}\label{section_main}

Our goal is to find an upper bound for the excess risk $\RP{L}{h_{D,s}}-\RPB{L}$. The idea is to split $X$ into two overlapping sets and to find a bound on the risks over these sets by using information on $P$. To this end, we denote the set of indices of cubes that intersect $X$ by 
\begin{align*}
J &:= \set{j \geq 1}{A_j \cap \left[-1,1 \right]^d  \neq \emptyset}.
\end{align*}
Next, we split this set into cubes that lie near the decision boundary and into cubes that are bounded away from the decision boundary. To be more precisely, we define, for $r>0$ and a version $\n$ for which the assumptions at the end of Section \ref{section_assump} hold, the set of indices of cubes near the decision boundary by
\begin{align*}
J_N^r &:=\set{\,j \in J}{\forall \,x \in A_j : \D_{\n}(x) \leq 3r}
\end{align*}
and the set of indices of cubes that are sufficiently bounded away by
\begin{align*}
J_F^r &:= \set{j \in J}{\forall \,x \in A_j :  \D_{\n}(x) \geq r}.
\end{align*}
Moreover, we write
\begin{align}
N_r &:= \bigcup_{j\in J_N^r} A_j,\label{def_A}\\
F_r &:= \bigcup_{j\in J_F^r} A_j \label{def_B}.
\end{align}
The next lemma shows that we are able to assign all $x \in A_j$ with $j \in J_F^r$
either to the class $X_{-1}$ or to $X_1$. Furthermore, we need to set geometric requirements to ensure that $X \subset N_r \cup F_r$.

\input{lemma_sets}

Lemma \ref{lemma_sets} ii) leads to a helpful splitting of the excess risk as the following lemma shows.

 \input{lemma_risksplit}

That means, we can bound the excess risk $\RP{L}{h_{D,s}}-\RPB{L}$ if we find bounds on the excess risks over the sets $N_r$ and $F_r$. For that purpose, we use an oracle inequality for empirical risk minimizer separately on both error terms, see \cite[Theorem 7.2]{StCh08}. This is possible, since the following lemma shows that, considering the loss $L_{T_J}$ for any set $T_J$ constructed as in (\ref{def_set_T}), the empirical histogram rule $h_{D,s}$ is still an empirical risk minimizer over $\F$.

\input{lemma_erm_subset}

Before we state our oracle inequality we discuss in a more detailed way the improvement that we gained by our separation technique described above. First, we make no approximation error on the set $F_r$, which consists of cells that are sufficiently bounded away from the decision boundary. This follows from the circumstance that $h_{D,s}$ learns correctly on those cells. We refer the reader to Part 1 of the proof of Lemma \ref{Hist_theorem_oracle_inequality} for details. Second, the main refinement arises from the fact that we achieve, under the condition that the decision boundary controls the noise from below,  a bound on $F_r$ of the form 
\begin{align*}
 \E_P (L \circ f-L \circ f^{\ast}_{L,P})^2&\leq V \cdot \E_P (L \circ f-L \circ f^{\ast}_{L,P})^{\th}
\end{align*}
with the best possible exponent, $\th=1$. Here, $V$ is a positive constant. The latter bound is known in the literature as variance bound. This bound plays an important part in the analysis of the risk terms since we have small variance if the right-hand side of the latter inequality is small. This relation is shown in detail in the next lemma.

\input{lemma_variance_bound}

We remark that the right-hand side of the variance bound on $F_r$ depends on the separation parameter $r$. This dependence is also reflected in the risk term on $F_r$. In particular, we show in Part 1 of the proof of Theorem \ref{Hist_theorem_oracle_inequality} by applying \cite[Theorem 7.2]{StCh08} on the risk term on the set $F_r$ that the improvements mentioned above lead to
\begin{align*}
\RP{L_{F_r}}{h_{D,s}}-\RPB{L_{F_r}} \leq \frac{32 \tilde{c}( 8^{d+1}s^{-d}+\t)}{r^{\g} n}
\end{align*}
with probability $P^n \geq 1-e^{-\t}$, where $\t\geq 1$ and $\tilde{c}$ is a positive constant. Whereas this error term increases for $r \rightarrow 0$, the error term on the set $N_r$ behaves exactly the opposite way, that is, it decreases for $r \rightarrow 0$. In fact, bounding the risk on $N_r$ requires additional knowledge of the behavior of $P$ in the vicinity of the decision boundary. By applying \cite[Theorem~7.2]{StCh08} on the risk on the set $N_r$ we show in Part 2 of the proof of  Theorem \ref{Hist_theorem_oracle_inequality} under the assumption that $P$ has ME $\a$ and MNE $\b$ that
\begin{align*}
\RP{L_{N_r}}{h_{D,s}}-\RPB{L_{N_r}}\leq 6(c_{\text{MNE}}s)^{\b}+ 4 \left( \frac{8V(\bar{c}rs^{-d} +\t)}{n}\right)^{\frac{\a+\g}{\a+2\g}}
\end{align*}
holds with probability $P^n \geq 1-e^{-\t}$. Here, $\bar{c}$ is a positive constant, $\t \geq 1$ and $V$ is the prefactor of the variance bound on $N_r$, shown in the second part of the proof. We refer the reader to the proof of Theorem \ref{Hist_theorem_oracle_inequality} for exact constants. If we balance the obtained risk terms over $N_r$ and $F_r$ with respect to $r$, we obtain the oracle inequality presented in the following theorem. For this purpose, we define the positive constant
\begin{align}\label{big_const}
\tilde{c}_{\a,\g,d}:= \left(\frac{16\g(\a+2\g) \cdot 8^{d+1} \max\lbrace c_{\text{LC}},2^{\g} \rbrace \cdot (\a+\g)^{-1}}{ \hat{c} ^{\frac{\a+\g}{\a+2\g}} } \right)^{\frac{\a+\g}{\a+\g+\g(\a+2\g)}},
\end{align}	
which depends on $\a,\g$ and $d$ and where 
\begin{align*}
\hat{c}:=32 \max \lbrace 12\mathcal{H}^{d-1}(\lbrace \n = 1/2 \rbrace),1\rbrace \cdot \max\left\lbrace 1, \frac{\a+\g}{\g}c_{\text{ME}}^{\frac{\a\g}{\a+\g}}\left( \frac{\g c_{\text{LC}}}{\a}\right)^{\frac{\a}{\a+\g}} \right\rbrace.
\end{align*}

\input{theorem_oracle_inequality}

The proof shows that the constants $c_{\a,\g,d}$ is given by
\begin{align}\label{constant_theorem_oracle_inequality}
c_{\a,\g,d}:=128 \cdot 8^{d+1}\max\lbrace c_{\text{LC}},2^{\g} \rbrace \cdot \max\left\lbrace \frac{\g(\a+2\g)}{\a+\g},1\right\rbrace \cdot \tilde{c}_{\a,\g,d}^{-\g}.
\end{align}

By choosing an appropriate sequence of $s_n$ in dependence of our data length $n$ and setting a constraint on the MNE $\b$ we state learning rates in the next theorem. Prior to that, we define with $\k:=(1+\g)(\a+\g)$ the positive constant 
\begin{align*}
\tilde{c}_{\a,\b,\g,\t,d}:=\left( \frac{d \cdot \k \cdot c_{\a,\g,d} \cdot \t^{\frac{(1+\g)(\a+\g)}{\k+\g^2}}}{6\b c_{\text{MNE}}^{\b}(\k+\g^2)}\right)^{\frac{\k+\g^2}{\b (\k+\g^2)+d\k}}
\end{align*}
that depends on $\a,\b,\g,\t$ and $d$ and where $c_{\a,\g,d}$ is the constant from (\ref{constant_theorem_oracle_inequality}).

\input{theorem_learning_rate}

The proof of the latter theorem shows that the constant $c_{\a,\b,\g,\t,d}$ is given by
\begin{align*}
c_{\a,\b,\g,\t,d}:=2\max\left\lbrace \frac{d \cdot \k}{\b(\k+\g^2)},1\right\rbrace c_{\a,\g,\d} \cdot\t^{\frac{\k}{\k+\g^2}}\cdot \tilde{c}_{\a,\b,\g,\t,d}^{-\frac{d\k}{\k+\g^2}}.
\end{align*}
Furthermore, we remark that the constraint on the MNE $\b$ in Theorem \ref{Hist_theorem_learning_rate} is set to secure that the chosen side length $s_n$ fulfils assumption (\ref{s_leq}). If we omit this constraint we have to chose another $s_n$. For this $s_n$ we would not be able to balance the two terms in the right-hand side of the excess risk in Theorem \ref{Hist_theorem_oracle_inequality}. Since our examples in Section \ref{section_comparison} fulfil this constraint we did not consider other choices of $s_n$.

To obtain the rates we have to know the parameters describing $P$. However, it is also possible to obtain the rates in Theorem \ref{Hist_theorem_learning_rate} by the following data splitting ansatz, whose concept is similar to the one described in \cite[Chapter~6.5]{StCh08}. Let $(S_n)$ be a sequence of finite subsets $S_n \subset (0,1]$. For a dataset $D:=((x_1,y_1),\ldots,(x_n,y_n))$ we define the sets
\begin{align*}
D_1&:=((x_1,y_1),\ldots,(x_k,y_k)),\\
D_2&:=((x_{k+1},y_{k+1}),\ldots,(x_n,y_n)),
\end{align*}
where $k:=\lfloor \frac{n}{2} \rfloor +1$ and $n\geq 4$. Then, we use $D_1$ as a training set and compute $h_{D_1,s}$ for $s \in S_n$ and use $D_2$ to determine $s^{\ast}_{D_2} \in S_n$ such that
\begin{align*}
s^{\ast}_{D_2}:= \argmin{s \in S_n}\ \Rx{L}{D_2}{h_{D_1,s}}.
\end{align*}
The resulting decision function is $h_{D_1,s^{\ast}_{D_2}}$ and a learning method producing this decision function is called training validation histogram rule (TV-HR). The following lemma shows that the TV-HR learns with the same rate as in Theorem \ref{Hist_theorem_learning_rate} without knowing the parameters describing $P$.

\input{theorem_learning_rate_TV}


\section{Comparison of Rates}\label{section_comparison}

In order to compare our rate obtained in Theorem \ref{Hist_theorem_learning_rate} to the ones known from \cite{AuTs07, BiCoDaDe14, KoKr07} and \cite{StCh08}, we set in the following reasonable sets of common assumptions. Besides our geometric assumption on $X$, namely
\begin{itemize}
\item[(i)] $X_0$ is $(d-1)$-rectifiable with $\mathcal{H}^{d-1}(X_0)>0$ and $X_0=\partial_X X_1=\partial_X X_{-1}$,
\end{itemize}
we make the following two assumptions on $P$:
\begin{itemize}
\item[(ii)] $P$ has \text{ME} $\a \in (0,\infty]$,
\item[(iii)] there exists a $\g \in [0,\infty)$ and constants $c_{\text{LC}},c_{\text{UC}}>0$ such that for all $x \in X$ we have
\begin{itemize}
\item[a)] $c_{\text{LC}} |2\n(x)-1|\geq \D_{\n}^{\g}(x)$,
\item[b)] $c_{\text{UC}} |2\n(x)-1|\leq \D_{\n}^{\g}(x)$.
\end{itemize}
\end{itemize}
Here, assumption $(iii)_{a}$ coincides with the definition in (\ref{def_lc}). Furthermore, assumption $(iii)_{b}$ shows that we have an upper control by $\D_{\n}$ on the noise,  which is up to a constant a kind of inverse to $(iii)_{a}$. Then, \cite[Lemma~8.17]{StCh08} shows under the assumptions $(ii)$ and $(iii)_b$ that $\text{P}$ has $\text{MNE}$ $\b=\a+\g$. Hence, we find by Theorem \ref{Hist_theorem_learning_rate} with $\k:=(1+\g)(\a+\g)$ and a suitable cell-width $s_n$ that $h_{D,s_n}$ learns with a rate with exponent
\begin{align*}
\tfrac{\beta(1+\g)(\a+\g)}{\beta\left[(1+\g)(\a+\g)+\g^ 2\right]+d(1+\g)(\a+\g)}&=\tfrac{(\a+\g)(1+\g)(\a+\g)}{(\a+\g)\left[(1+\g)(\a+\g)+\g^ 2\right]+d(1+\g)(\a+\g)}\\
&=\tfrac{(1+\g)(\a+\g)}{(1+\g)(\a+\g)+\g^2+d(1+\g)}.
\end{align*}
A simple transformation shows that this exponent equals
\begin{align}\label{calc_rate}
\begin{split}
\tfrac{(1+\g)(\a+\g)}{(1+\g)(\a+\g)+\g^2+d(1+\g)}=\tfrac{\a+\g}{\a+\g+\tfrac{\g^2}{1+\g}+d}=\tfrac{\a+\g}{\a+2\g+\tfrac{\g^2}{1+\g}+d-\g}=\tfrac{\a+\g}{\a+2\g+d-\tfrac{\g}{1+\g}}.
\end{split}
\end{align}

First, we compare the rate with exponent (\ref{calc_rate}) to the rate achieved by support vector machines (SVM) for the hinge loss by assuming that $(i),(ii)$ and $(iii)$ hold. For this purpose, \cite[Chapter~8.3 (8.18)]{StCh08} shows that the best possible rate for SVMs using Gaussian kernels is obtained by
\begin{align*}
n^{-\frac{\a+\g}{\a+2\g+d}+\rho},
\end{align*}
where $\rho>0$ is an arbitrary small number. Hence, our rate in (\ref{calc_rate}) is better by $-\frac{\g}{1+\g}$ in the denominator. For the typical value of $\g=1$, indicating a moderate control of noise by the decision boundary, our rate is better by $-1/2$ in the denominator. 

Second, we compare our rates to the ones for certain plug-in classifiers, see \cite{AuTs07, KoKr07}, and to the rates obtained by the classification algorithms, described in \cite{BiCoDaDe14}. In the cases of \cite{AuTs07} and \cite{BiCoDaDe14} the authors assume that $P$ has a noise exponent (NE) $q \in [0, \infty ]$, that is, that there exist a constant $c_{\text{NE}}>0$ such that
\begin{align}\label{NE}
P_X(\lbrace x\in X: |2\n(x) -1|<\e \rbrace ) \leq (c \e)^q 
\end{align}
for all $\e > 0$, c.f.\ \cite[Definition~8.22]{StCh08}. Since (\ref{NE}) measures the amount of critical noise and does not locate noise we call this exponent noise exponent in contrast to \citep{MaNe06} and the mentioned authors, who call this exponent margin exponent. The authors of \cite{KoKr07} assume a weaker version of (\ref{NE}) on $P$, however, the latter implies this weak version, see \cite[Section~2]{DoeGyWa15}. We compare our rates under a different assumption set as in the first comparison to SVMs. To this end, we impose in addition to $(i),(ii)$ and $(iii)_a$ that
\begin{itemize}
\item[(iv)] $\n$ is H\"older-continuous for some $\g \in (0,1]$.
\end{itemize}
Then, we find under condition $(iv)$ with Lemma \ref{Appendix1} that assumption $(iii)_b$ is fulfilled with exponent $\g$ and thus we assume in the following that $(iii)_a$ holds for the same $\g$.
Note that in $(iv)$ we have $\g \in (0,1]$, whereas in the case of $(iii)$ we have $\g \in [0,\infty)$. Moreover, under assumptions $(ii)$ and $(iii)_a$ we find with \cite[Exercise~8.5]{StCh08} that the noise exponent in (\ref{NE}) holds with 
\begin{align}\label{q_is_a_divided_g}
q=\frac{\a }{\g}.
\end{align}
By assuming $(i),(ii), (iii)_a$ and $(iv)$ our rate yields the same exponent as in (\ref{calc_rate}), that is 
\begin{align}\label{calc_rate_gtilde}
\tfrac{\a+\g}{\a+2\g+d-\tfrac{\g}{1+\g}}.
\end{align}
Furthermore, the plug-in classifiers based on kernel, partitioning or nearest neighbor regression estimates shown in \cite[Theorem~1,~3 and 5]{KoKr07} yield under these assumptions and thus in particular with (\ref{q_is_a_divided_g}) the rate
\begin{align}\label{rate1_KoKr}
n^{-\frac{\a+\g}{\a+3\g+d}},
\end{align}
such that our rate is better by $-\frac{\g(2+\g)}{1+\g}$ in the denominator. The authors were able to improve the rate given in (\ref{rate1_KoKr}) by making in addition the assumption that $P_X$ has a density with respect to the Lebesgue measure, which is bounded away from zero, see \cite[Theorem~2,~4 and 6]{KoKr07}. Under this condition and  $(i),(ii), (iii)_a$ and $(iv)$ the classifiers yield the rate
\begin{align*}
n^{-\frac{\a+\g }{2\g +d}}.
\end{align*}
Hence, our rate with exponent (\ref{calc_rate}) is better if our margin exponent $\a$ fulfils $\a<\frac{\g}{1+\g}$. We have small margin exponent $\a$, for example, if we have much mass around the decision boundary, that is, the density is unbounded in this region. We remark that the authors obtained rates under the H\"older assumption $(iv)$, a weak margin assumption, and improved them as discussed above by making the assumption that $P_X$ has a density which is bounded away from zero.

Next, we compare our rates to the ones obtained by the classifier resulting from the classification method given in \cite[Section~5]{BiCoDaDe14}. Therefore, we consider in addition to $(i), (iii)_a$ and $(iv)$ for example that
\begin{itemize}
\item[(v)] $P_X$ is the uniform distribution. 
\end{itemize}
Under the condition that $(i)$ and $(v)$ hold, we find with Lemma \ref{lemma_hausdorff} that assumption $(ii)$ is fulfilled for $\a=1$. Then, we obtain in (\ref{q_is_a_divided_g}) that $q=\frac{1}{\g}$. Again, we find with Lemma \ref{Appendix1} that assumption $(iii)_b$ is fulfilled with exponent $\g $ and assume again that $(iii)_a$ holds for the same $\g $.  Hence, the conditions $(i)$ and $(iii)_a, (iv)$ and $(v)$ yield in (\ref{calc_rate_gtilde}) a rate with exponent
\begin{align}\label{rate_meins}
\tfrac{1+\g }{1+2\g +d-\tfrac{\g }{1+\g }}
\end{align}
for our method. Furthermore, \cite[Corollary~5.2(ii)]{BiCoDaDe14} shows that the classifier mentioned in \cite[Section~5]{BiCoDaDe14} yields the rate
\begin{align}\label{rate1_BiCoDaDe}
\left(\frac{(\log n)^{\frac{1}{2+d}}}{n}\right)^{\frac{1+\g }{2\g +d}}.
\end{align} 
Hence, our rate is worse by $\frac{1}{1+\g}$. However, the rate given in (\ref{rate1_BiCoDaDe}) is also comparable under a more generic assumption set in which we do not fix an example of $P_X$. Indeed, if we assume the conditions $(i),(ii),(iii)_a$ and $(iv)$, then, our rate with exponent (\ref{calc_rate_gtilde}) holds and \cite[Corollary~5.2(i)]{BiCoDaDe14} shows that their classifier obtains the rate
\begin{align}\label{rate2_BiCoDaDe}
\left(\frac{\log n}{n}\right)^{\frac{\a+\g }{\a+2\g +d}}.
\end{align}
Thus, our rate with exponent (\ref{calc_rate_gtilde}) is again better by $-\frac{\g }{1+\g }$ in the denominator.

Finally, we compare our rates to the ones obtained for the plug-in classifier defined by \cite[(4.1) with $p=\infty$]{AuTs07} under the conditions $(i),(iii)_a, (iv)$ and 
\begin{itemize}
\item[(vi)] $P_X$ has a uniformly bounded density. 
\end{itemize}
Analogously as above one can show with $(i)$ and $(vi)$ with Lemma \ref{lemma_hausdorff} that we have MNE $\a=1$. 
Under these conditions our rate with exponent (\ref{rate_meins}) holds. The classifier in \cite[(4.1) with $p=\infty$]{AuTs07} achieves the rate
\begin{align}\label{rate_AuTs}
n^{-\frac{1+\g }{1+2\g +d}}
\end{align}
in expectation and we find that our rate is better by $-\frac{\g }{1+\g}$ in the denominator. We remark at this point that  \citep[Theorem 4.1 and 4.3]{AuTs07} proved that the classifier achieves this rate under a different assumption set, namely under $(iv),(vi)$ and the assumption that $P$ has NE $q \in [0, \infty ]$. The classifier then achieves the rate
\begin{align}\label{rate_AuTs_allg}
n^{-\frac{\g(q+1)}{\g(q+2)+d}}
\end{align}
and for this set of assumptions the rate is optimal (in a minimax sense). Our assumptions, namely $(i),(iii)_a, (iv)$ and $(vi)$ imply the assumptions of $[1]$, but, this is not a contradiction since our assumptions are a subset of the assumptions of $[1]$.

Our improvement arises from assumption $(iii)_a$ since it forces critical noise ($\n \approx 1/2$) to be located close to the decision boundary and as we will see down below this assumption has an essential influence on the NE $q$. To be more precisely, there are two sources for slow learning rates. The first one is the approximation error around the decision boundary, the second one is the existence of critical noise. Assumption $(iii)_a$ forces both to be in the same region such that both effects cannot independently occur, which in turn leads to better rates compared to \citep{AuTs07}. In other words, with assumption $(iii)_a$ we exclude distributions that have regions of critical noise that are far away from the decision boundary. Be aware that this does not mean that we consider only distributions without noisy regions bounded away from the decision boundary. In Fig.\ \ref{figure1} we present two examples which make this situation more clear. Areas of noise that are, for example, located in the set $X_1$ in Fig.\ \ref{figure1} (a) resp.\ in the set $X_{-1}$ in Fig.\ \ref{figure1} (b) are still allowed under $(iii)_a$ whereas the areas of critical noise in the particular other set are permitted.
\begin{figure}
\begin{subfigure}[h]{0.7\linewidth}
\includegraphics[width=\linewidth]{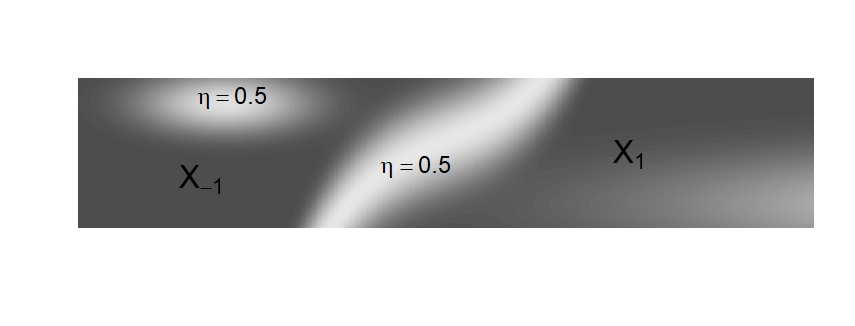}
\caption{The brighter the region the more closer $\n$ is to $1/2$.  The decision boundary (within the bright stripe) is located in the middle of the picture. Far away from the decision boundary we locate in both sets $X_{-1}$ and $X_1$ noise (upper left and lower right corner), but only in the upper left corner the critical level $\n=1/2$ is reached.}
\end{subfigure}
\hfill
\begin{subfigure}[h]{0.7\linewidth}
\includegraphics[width=\linewidth]{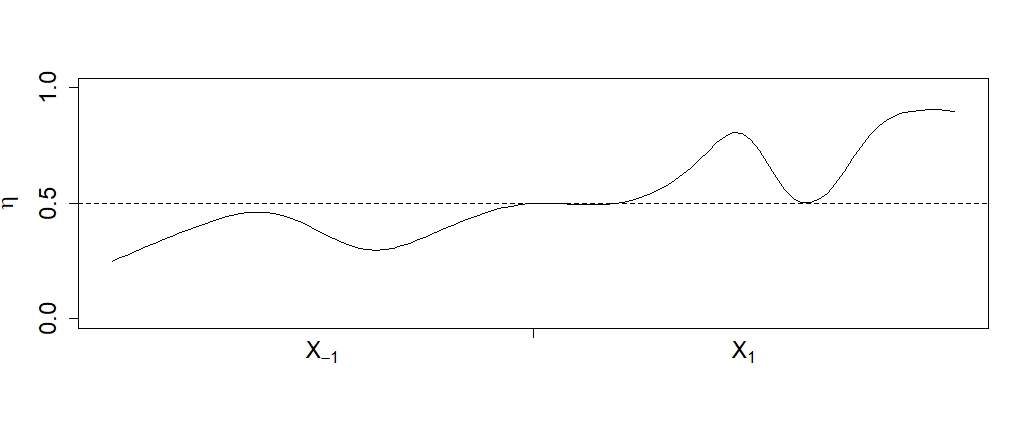}
\caption{Far away from the decision boundary, in $X_1$, we locate a region of critical noise. In  $X_{-1}$ we have noise too, but the critical level $\n=1/2$ is not reached.}
\end{subfigure}%
\caption{Examples of $\eta$ with regions of critical noise ($\n \approx 1/2$) far away from the decision boundary. The size of these regions has an essential influence on the NE $q$. Since assumption $(iii)_a$ disallow areas of critical noise far away from the decision boundary we obtain a better noise exponent and hence, better rates. Note that areas of noise as in (a) in the lower right corner or in (b) in the lower left corner are still allowed.}
\label{figure1}
\end{figure} 
 
To make this heuristic argument more precise we take a look at Theorem \ref{Hist_theorem_oracle_inequality} and its proof and show that if we omit assumption $(iii)_a$ and thus consider the assumptions taken in \citep{AuTs07}, we match the optimal rate in (\ref{rate_AuTs_allg}). To this end, we consider in addition to $(i)$ the above mentioned assumptions of \citep{AuTs07}, that is $(iv),(vi)$ and the NE $q$. Since we do not assume $(iii)_a$ we cannot use this assumption to obtain a variance bound 
on the set $F_r$, which is bounded away from the decision boundary (e.g., Lemma \ref{Lemma_var_bound}). Hence, the separation technique we used in the proof would make no sense any more, but we are able to bound the excess risk on the whole set by Part 2 in the proof of Theorem \ref{Hist_theorem_oracle_inequality}, where we bounded the excess risk on the set $N_r$ that is close to the decision boundary. This situation corresponds to the fact that our set $F_r$ is empty (letting go $r \to \infty$). There are two points that change in Part 2 of the proof. First, we have no variance bound as in (\ref{calc_var_bound}), but we can apply \cite[Theorem~8.24]{StCh08}, a general variance bound, and obtain
\begin{align*}
\E_P (h^N_{f_0})^2 \leq \tilde{V} \left(\E_P h^N_{f_0}\right)^{\th},
\end{align*}
where $\th:=\frac{q}{q+1}$ and where $\tilde{V}$ is a positive constant. Second, we can bound the cardinality $|\F|$ in (\ref{calc_log_bound}) as in Part 1 and yield with some calculations as in Part 3 for the overall excess risk 
\begin{align*}
\RP{L}{h_{D,s}}-\RPB{L}
 &\leq 6\left(c_{\text{MNE}}s \right)^{\b}+ c \left( \frac{\t}{s^{d}n}\right)^{\frac{1}{2-\th}},
\end{align*}
where $c>0$ is a constant and $\t\geq 1$. Then, minimizing over $s$ yields for our learning method the rate
\begin{align*}
n^{-\frac{\g(q+1)}{\g(q+2)+d}},
\end{align*}
which matches the in \citep[Theorem 4.1 and 4.3]{AuTs07} proven optimal rate (\ref{rate_AuTs_allg}). We further remark that instead of the H\"older assumption $(iv)$ the weaker assumption $(iii)_b$ is sufficient for Theorem \ref{Hist_theorem_oracle_inequality} and the above modified one. 

If we consider now in addition to the assumptions in \citep{AuTs07} that $(i)$ and $(iii)_a$ hold, we find that our rate improves immediately since we can directly apply Theorem \ref{Hist_theorem_oracle_inequality} and thus obtain exactly the rate with exponent (\ref{rate_meins}) which is better by $-\frac{\g }{1+\g }$. In summary, taking in consideration assumption $(iii)_a$ influences the noise exponent in a good way since we exclude distributions that have critical noise far away from the decision boundary. This leads to better learning rates.

Finally, we remark that for our results as well as for the results from \cite{AuTs07,BiCoDaDe14, KoKr07} and \cite{StCh08} less assumptions are sufficient and in the comparisons above we tried to formulate reasonable sets of common assumptions.

\section{Proofs}
\label{section_proofs}

\input{hist_proofs}
\appendix

\section{Appendix}

\input{hist_appendix}


\bibliographystyle{imsart-nameyear}

\bibliography{ing_literature}

\end{document}

%% file: ing_macros.tex
\newtheorem{theorem}{Theorem}[section]
\newtheorem{lemma}[theorem]{Lemma}

\newenvironment{proofof}[1]{\noindent{\bf Proof of #1:}}{\qed\medskip}

\renewcommand*\proofname{Proof}
\makeatletter
\renewenvironment{proof}[1][\proofname]{ \par
\pushQED{\qed}%
\normalfont \topsep0em\@plus0em\relax 
\trivlist
\item[ \hskip0.1em 
\bfseries 
#1\@addpunct{:}]\normalfont 
}{%
\popQED\endtrivlist\@endpefalse
}
\makeatother

\newcommand{\mycdot}{\,\cdot\,}
\newcommand{\Leq}{\leq {}}
\newcommand{\Geq}{\geq {}}
\newcommand{\set}[2]{\lbrace \,{#1}\,|\,{#2}\,\rbrace}

\newcommand{\R}{\mathbb{R}}

\newcommand{\F}{\mathcal{F}}

\newcommand{\E}{\mathbb{E}}

\newcommand{\dx}[1]{\hspace*{0.25ex}d\hspace*{-0.15ex}#1}
\newcommand{\dxy}[2]{\dx{#1}(#2)}

\renewcommand{\a}{\alpha}
\renewcommand{\b}{\beta}
\newcommand{\g}{\gamma}

\renewcommand{\d}{\delta}
\newcommand{\D}{\Delta}
\newcommand{\e}{\varepsilon}

\newcommand{\z}{\zeta}
\newcommand{\n}{\eta}

\renewcommand{\k}{\kappa}
\newcommand{\lb}{\lambda}

\renewcommand{\t}{\tau}
\renewcommand{\th}{\theta}

\newcommand{\argmin}[1]{\underset{#1}{\operatorname{arg\,min}}}

\newcommand{\eins}{\boldsymbol{1}}

\newcommand{\inorm}[1]{\Vert #1 \Vert_\infty}

\newcommand{\RP}[2]{{{\cal R}_{#1,P}(#2)}}
\newcommand{\RPB}[1]{{{\cal R}_{#1,P}^{*}}}

\newcommand{\RPxB}[2]{{{\cal R}_{#1,P,#2}^*}}
\newcommand{\Rx}[3]{{{\cal R}_{#1,#2}(#3)}}

\newcommand{\lclass}{{{L}_{\mathrm{class}}}}


%% file: lemma_hausdorff.tex
\begin{lemma}\label{lemma_hausdorff}
Let $X:=\left[-1,1 \right]^d$ and $P$ be a probability measure on $X \times \{-1,1\}$ with fixed version $\n \colon X \to [0,1]$ of its posterior probability. Moreover, let $\lb^d$ be the $d$-dimensional Lebesgue measure and $\mathcal{H}^{d-1}$ be the $(d-1)$-dimensional Hausdorff measure on $\mathbb{R}^d$. Furthermore, let $X_0=\partial_X X_1$ with $\mathcal{H}^{d-1}(X_0)>0$ and let $X_0$ be $(d-1)$-rectifiable. Then, there exists a $\d^{\ast}>0$ such that for all $\d \in (0, \d^{\ast} ]$ we have 
\begin{align*}
\lb^d(\set{x\in X}{\D_{\n}(x) \leq \d}) \leq 4 \mathcal{H}^{d-1}(\set{x \in X}{\n(x)=1/2}) \cdot \d.
\end{align*}
\end{lemma}

%% file: lemma_sets.tex
\begin{lemma}\label{lemma_sets}
Let $\mathcal{A}=(A_j)_{j \geq 1}$ be a partition of $\mathbb{R}^d$ into cubes of side length $s \in (0,1]$ and let $X:=\left[-1,1 \right]^d$. For $r \geq s/2$ define the sets $N_r$ and $F_r$ by (\ref{def_A}) and (\ref{def_B}). Then, the following statements are true:
\begin{itemize}
\item[\label{either}{i)}] We have either $A_j \cap X_1 = \emptyset$ or $A_j \cap X_{-1} = \emptyset$ for $j \in J_F^r$.
\item[\label{xsubsetcup}{ii)}] If $X_0=\partial_X X_1=\partial_X X_{-1}$, we have $X \subset N_r \cup F_r$.
\end{itemize}
\end{lemma}

%% file: lemma_risksplit.tex
\begin{lemma}
Under the assumptions of Lemma \ref{lemma_sets} ii) we have
\label{lemma_risksplit}
\begin{align*}
\begin{split}
&\RP{L}{h_{D,s}}-\RPB{L}\\
 &\leq \left( \RP{L_{N_r}}{h_{D,s}} -\RPB{L_{N_r}}\right)+ \left( \RP{L_{F_r}}{h_{D,s}}-\RPB{L_{F_r}} \right).
 \end{split}
\end{align*}
\end{lemma}

%% file: lemma_erm_subset.tex
\begin{lemma}\label{lemma_erm_subset}
Consider for an arbitrary index set $J \subset \lbrace \,1,\ldots,m \,\rbrace$ the set $T_J:= \bigcup_{j\in J} A_j$ and the related loss $L_{T_J}: X \times Y \times \mathbb{R} \rightarrow [0, \infty)$ defined in (\ref{loss_A}). Then, the empirical histogram rule $h_{D,s}$ is an empirical risk minimizer over $\F$ for the loss $L_{T_J}$, that means
\begin{align*}
\Rx{L_{T_J}}{D}{h_{D,s}}=\inf_{f\in \F}\Rx{L_{T_J}}{D}{f}.
\end{align*}
\end{lemma}

%% file: lemma_variance_bound.tex
\begin{lemma}\label{Lemma_var_bound}
Let $X:=\left[-1,1 \right]^d$ and $P$ be a probability measure on $X \times \{-1,1\}$ with fixed version $\n \colon X \to [0,1]$ of its posterior probability. Assume that the associated distance to the decision boundary  $\D_{\n}$ controls the noise from below by the exponent $\g \in [0,\infty)$ and consider for some fixed $r >0$ the set $F_r$, defined in $(\ref{def_B})$. Furthermore, let $L:=L_{class}$ be the classification loss and let $f^{\ast}_{L,P}$ be a fixed Bayes decision function. Then, for all measurable $f \colon X \to \lbrace -1,1 \rbrace$ we have
\begin{align*}
 \E_P (L_{F_r} \circ f-L_{F_r} \circ f^{\ast}_{L,P})^2&\leq \frac{c_{LC}}{r^{\g}}  \E_P (L_{F_r} \circ f-L_{F_r}\circ f^{\ast}_{L,P}).
\end{align*}
\end{lemma}


%

%% file: theorem_oracle_inequality.tex
\begin{theorem}\label{Hist_theorem_oracle_inequality}
Let $\mathcal{A}=(A_j)_{j \geq 1}$ be a partition of $\mathbb{R}^d$ into cubes of side length $s \in (0,1]$. Let $X:=\left[-1,1 \right]^d$ and $P$ be a probability measure on $X \times \{-1,1\}$ with fixed version $\n \colon X \to [0,1]$ of its posterior probability. Assume that the associated distance to the decision boundary  $\D_{\n}$ controls the noise from below by the exponent $\g \in [0,\infty)$ and assume as well that P has MNE $\beta \in (0, \infty]$ and ME $\a \in (0,\infty]$. Furthermore, let $X_0=\partial_X X_1=\partial_X X_{-1}$ with $\mathcal{H}^{d-1}(X_0)>0$ and let $X_0$ be $(d-1)$-rectifiable. Let $L$ be the classification loss and let for fixed $n \geq 1$ and $\tau \geq 1$ the bounds 
\begin{align}\label{s_leq}
s \leq \tilde{c}_{\a,\g,d}^{\frac{(1+\g)(\a+\g)+\g^2}{(1+\g)(\a+\g)+\g^2+d\g}}\left(\frac{\t}{n} \right)^{\frac{\g}{(1+\g)(\a+\g)+\g^2+d\g}},
\end{align}
and 
\begin{align}\label{s_geq}
s^dn \geq \t \left( \frac{\tilde{c}_{\a,\g,d}}{\min\lbrace \frac{\d^{\ast}}{3},1\rbrace} \right)^{\frac{(1+\g)(\a+\g)+\g^2}{\g}}
\end{align}
be satisfied, where the constant $\tilde{c}_{\a,\g,d}$ is defined by (\ref{big_const}) and the constant $\d^{\ast}>0$ is the one of Lemma \ref{lemma_hausdorff}. Then, there exists a constant $c_{\a,\g,d}>0$ such that
\begin{align}\label{oracle_theorem}
\RP{L}{h_{D,s}}-\RPB{L}&\leq 6\left(c_{\text{MNE}}s \right)^{\b}+ c_{\a,\g,d} \left( \frac{\t}{s^dn} \right)^{\frac{ (1+\g)(\a+\g)}{(1+\g)(\a+\g)+\g^2}}
\end{align}
holds with probability $P^n \Geq 1-2e^{-\t}$, where the constant  $c_{\a,\g,d}$ only depends on $\a,\g$ and $d$.
\end{theorem}

%% file: theorem_learning_rate.tex
\begin{theorem}\label{Hist_theorem_learning_rate}
Assume that $X$ and $P$ satisfy the assumptions of Theorem \ref{Hist_theorem_oracle_inequality} for $\b \leq \g^{-1}\k$, where $\k:=(1+\g)(\a+\g)$. In addition, assume that the side length $s_n$ in Theorem \ref{Hist_theorem_oracle_inequality} is given by
\begin{align*}
s_n=\tilde{c}_{\a,\b,\g,\t,d} n^{-\frac{\k}{\b(\k+\g^2)+d\k}}.
\end{align*}
Then, there exists a constant $c_{\a,\b,\g,\t,d}>0$ such that for all $n \geq n_0$
\begin{align*}
\RP{L}{h_{D,s_n}}-\RPB{L} \leq c_{\a,\b,\g,\t,d} n^{-\frac{\b\k}{\b(\k+\g^ 2)+d\k}}
\end{align*}
holds with probability $P^n \geq 1-2e^{-\t}$, where $n_0$ and the constant $c_{\a,\b,\g,\t,d}$ only depend on $\t,\a,\b,\g$ and $d$. 
\end{theorem}

%% file: theorem_learning_rate_TV.tex
\begin{theorem}\label{theorem_learning_rate_TV}
Assume that $X$ and $P$ satisfy the assumptions of Theorem \ref{Hist_theorem_oracle_inequality} for $\b \leq \g^{-1}\k$, where $\k:=(1+\g)(\a+\g)$. Let $S_n$ be a finite subset of $(0,1]$ such that $S_n$ is a $n^{-1/d}$-net of $(0,1]$. Assume that the cardinality of $S_n$ grows at most polynomially in $n$. Then, the TV-HR learns with rate
\begin{align*}
n^{-\frac{\b\k}{\b(\k+\g^2)+d\k}}.
\end{align*}
\end{theorem}

%% file: hist_proofs.tex
\begin{proofof}{Lemma \ref{lemma_hausdorff}}
For a set $T \subset X$ and $\d>0$ we define as in \cite{Steinwart15a} the sets
\begin{align*}
T^{+\d}&:= \set{x \in X}{d(x,T)\leq \d},\\
T^{-\d}&:= X \setminus (X \setminus T)^{+\d}.
\end{align*}
Since $X_1:=\set{x\in X}{\n(x)\leq 1/2}$ is bounded and measurable, 
we find with \cite[Lemma~A.10.3]{Steinwart15a} and the proof of \cite[Lemma~A.10.4(ii)]{Steinwart15a} that there exists a $\d^{\ast}>0$, such that for all $\d \in (0,\d^{\ast}]$ we have
\begin{align}\label{app}
\lb^{d}(X_1^{+\d} \setminus X_1^{-\d}) \leq 4 \mathcal{H}^{d-1}(\partial_X X_1) \cdot \d = 4\mathcal{H}^{d-1}(X_0) \cdot \d.
\end{align} 
Next, we show that
\begin{align}\label{tm}
\set{x \in X}{\D_{\n}(x) \leq \d} \subset X_1^{+\d} \setminus X_1^{-\d}\cup X_0.
\end{align}
For this purpose, we remark that according to (\ref{delta}) we have
\begin{align*}
&\set{x \in X}{\D_{\n}(x) \leq \d}\\
&= \set{x\in X_{1}}{d(x,X_{-1})\leq \d} \cup \set{x\in X_{-1}}{d(x,X_{1})\leq \d} \cup X_0.
\end{align*}
Let us first show that $\set{x\in X_{1}}{d(x,X_{-1})\leq \d} \subset X_1^{+\d} \setminus X_1^{-\d}$. To this end, consider an $x\in X_{1}$ with $d(x,X_{-1})\leq \d$, where we check at once that $x \in X_1^{+\d}$. Now, assume that $x \in X_1^{-\d}=X \setminus (X \setminus X_1)^{+\d}$. Then, we find that 
$x \notin (X \setminus X_1)^{+\d}$ such that $d(x,X \setminus X_1)=d(x,X_{-1}\cup X_0)>\d$. Hence, $x \notin X_1^{-\d}$.
Next, let us show that $\set{x\in X_{-1}}{d(x,X_{1})\leq \d} \subset X_1^{+\d} \setminus X_1^{-\d}$. To this end, consider an $x \in X_{-1}$ with $d(x,X_{1})\leq \d$. Then, it is clear that $x \in X_1^{+\d}$ by definition of $X_1^{+\d}$. Furthermore, $x \notin X_1^{-\d}$ since $X_1^{-\d}=X \setminus (X_{-1})^{+\d} \subset X_1$. Having showed (\ref{tm}), we find together with the fact that $\lb^d (X_0)=0$ since $X_0$ is $(d-1)$- rectifiable that
\begin{align*}
\lb^{d}( \set{x \in X}{\D_{\n}(x) \leq \d} )\leq \lb^{d}(X_1^{+\d} \setminus X_1^{-\d}).
\end{align*}
Finally, with (\ref{app}) and the fact that $X_0=\partial_X X_1$ we find that
\begin{align*}
\lb^{d}( \set{x \in X}{\D_{\n}(x) \leq \d} ) &\leq \lb^{d}(X_1^{+\d} \setminus X_1^{-\d})\\
&\leq 4\mathcal{H}^{d-1}(X_0) \cdot \d\\
&= 4\mathcal{H}^{d-1}(\set{x \in X}{\n(x)=1/2}) \cdot \d
\end{align*}
for all $\d \in (0, \d^{\ast}]$.
\end{proofof}

\begin{proofof}{Lemma \ref{lemma_sets}}
\begin{itemize}
\item[i)] We assume for $A_j$ with $j \in J_F^r$ that we have an $x_1 \in A_j \cap X_1 \neq \emptyset$ and an $x_{-1} \in A_j \cap X_{-1} \neq \emptyset$. Then, the connecting line $\overline{x_{-1}x_1}$ from $x_{-1}$ to $x_1$ is contained in $A_j$ since $A_j$ is convex and we have $\inorm{x_{-1}-x_1}\leq s$. Moreover, since $\D_{\n}(x) \geq r$ for all $x\in F_r$ we have that $x \in X_1 \cup X_{-1}$. Next, pick an $m>1$ such that
\begin{align*}
t_0=0, \qquad t_m=1, \qquad t_i=\frac{i}{m}
\end{align*}
and
\begin{align*}
x_i:=t_ix_{-1} + (1-t_i)x_1
\end{align*}
for $i=0,\ldots,m$. Clearly, $x_i \in \overline{x_{-1}x_1}$ and $x_i \in X_{-1}\cup X_1$. Since $x_0 \in X_1$ and $x_m \in X_{-1}$, there exists an $i$ with $x_i \in X_1$ and $x_{i+1} \in X_{-1}$ and we find that
\begin{align*}
\inorm{x_i-x_{i+1}} \geq \D_{\n}(x_i) \geq r.
\end{align*} 
On the other hand,
\begin{align*}
\inorm{x_i-x_{i+1}}=\frac{1}{m}\inorm{x_{-1}-x_{1}}\leq\frac{s}{m}\leq \frac{2r}{m}
\end{align*}
such that $r \leq \frac{2r}{m}$, which is not true for $m \geq 3$. Hence, we can not have an $x_1 \in A_j \cap X_1 \neq \emptyset$ and an $x_{-1} \in A_j \cap X_{-1} \neq \emptyset$ for $j \in J_F^r$.
\item[ii)] We define the set of indices
\begin{align*}
J_C^r := \set{j \in J}{\exists \,\tilde{x} \in A_j : \D_{\n}(\tilde{x}) < r}
\end{align*}
 and define the set
\begin{align*}
C_r &:= \bigcup_{j\in J_C^r} A_j.
\end{align*}
Since $X \subset F_r \cup C_r$, it suffices to show that $C_r \subset N_r$. To show the latter we fix an $x \in C_r$.
If $x \in X_0$ we immediately have $\D_{\n}(x)=0< 3r$, hence we assume w.l.o.g.\ that $x \in X_1$. Then, there exists a $j \in J_C^r$ such that $x \in A_j$. Furthermore, there exists an $x^{\ast} \in A_j$ with $\D_{\n}(x^{\ast})<r$ and we find with $X_0=\partial_X X_1=\partial_X X_{-1}$ that
\begin{align*}
\D_{\n}(x)&=\inf_{x'\in X_{-1}} \inorm{x- x'}\\
			&\leq \inf_{x'\in X_{-1}} \left( \inorm{x- x^{\ast}}+\inorm{x^{\ast}- x'} \right) \\
			&\leq s+\D_{\n}(x^{\ast})\\
			&< s+r,
\end{align*}
where $\inorm{\cdot}$ is the supremum norm in $\R^d$. Since $s \leq 2r$, it follows that $\D_{\n}(x) \leq 3r$ and therefore $x \in N_r$.
\end{itemize}
\end{proofof}

\begin{proofof}{Lemma \ref{lemma_risksplit}}
Under the assumptions of Lemma \ref{lemma_sets} ii) we find that $X \subset N_r \cup F_r$. Since the excess risk is non-negative we then have
\begin{align*}
&\Rx{L}{P}{f}-\RPB{L}\\
&=\int_{X \times Y} L(y,f(x)-f^{\ast}(x)) \,dP(x,y)\\
&=\int_{X}\int_{Y}L(y,f(x)-f^{\ast}(x)) \,P(dy|x)dP_X(x)\\
&\leq \int_{N_r \cup F_r}\int_{Y}L(y,f(x)-f^{\ast}(x)) \,P(dy|x)dP_X(x)\\
&=\int_{N_r}\int_{Y}L(y,f(x)-f^{\ast}(x)) \,P(dy|x)dP_X(x)\\
&\quad+\int_{F_r}\int_{Y}L(y,f(x)-f^{\ast}(x)) \,P(dy|x)dP_X(x)\\
&=\left(\RP{L_{N_r}}{h_{D,s}} -\RPB{L_{N_r}}\right)+ \left( \RP{L_{F_r}}{h_{D,s}}-\RPB{L_{F_r}}\right).
\end{align*}\end{proofof}

\begin{proofof}{Lemma \ref{lemma_erm_subset}}
For $f \in \mathcal{F}$ we have
\begin{align*}
&\Rx{L_{T_J}}{D}{f}&\\
&=\int_{X \times Y} L_{T_J}(x,y,f(x)) \,dD(x,y)\\
&=\sum_{j \in J} \int_{A_j \times Y} \lclass(y,f(x)) \,dD(x,y).
\end{align*}
Next, we take a closer look at the risk on a single cell $A_j$ for $j \in J$. That is, 
\begin{align*}
\int_{A_j \times Y} \lclass(y,f(x))  \,dD(x,y)=\frac{1}{n}\sum_{i=1}^n \eins_{A_j}(x_i)\eins_{y_i \neq c_j},
\end{align*}
where $c_j \in \lbrace -1,1\rbrace$ is the label of the cell $A_j$. The risk on a cell is the smaller the less often we have $y_i \neq c_j$ such that the best classifier on a cell is the one which decides by majority. This is true for the histogram rule by definition. Since the risk is zero on $A_j$ with $j \not\in J$, the histogram rule minimizes the risk with respect to $L_{T_J}$.
\end{proofof}

\begin{proofof}{Lemma \ref{Lemma_var_bound}}
We define $h _f:=L_{F_r} \circ f-L_{F_r} \circ f^{\ast}_{L,P}$ for a measurable $f \colon X \to \lbrace -1,1 \rbrace$. Since $(L_{F_r}\circ f-L_{F_r}\circ f^{\ast}_{L,P})^2=\eins_{F_r}\frac{|f-f^{\ast}_{L,P}|}{2}$ we obtain 
\begin{align*}
&\E_P (h_f -\E_P h_f)^2\\
 &\leq \E_P (h_f)^2\\
 &=\E_P (L_{F_r}\circ f-L_{F_r}\circ f^{\ast}_{L,P})^2\\
 &=\frac{1}{2}\E_P \eins_{F_r}|f-f^{\ast}_{L,P}|.
\end{align*}
For $x \in F_r$ we have $\D_{\n}(x) \geq r$ and thus we find with our lower-control assumption that
\begin{align*}
r^{\g} \leq \D_{\n}^{\g}(x) \leq c_{LC} |2\n(x)-1|
\end{align*}
and therefore
\begin{align*}
1 \leq c_{LC}r^{-\g}|2\n(x)-1|.
\end{align*}
By using $\eins_{F_r}\frac{|f-f^{\ast}|}{2}= \eins_{(X_{-1} \triangle \lbrace f < 0 \rbrace)\cap F_r}$, where $\triangle$ denotes the symmetric difference defined by $C \triangle D := (C \setminus D) \cup (D \setminus C)$ for sets $C,D \subset X$ and by using Lemma \ref{Appendix1} we obtain for the variance bound
\begin{align*}
\E_P (h_{f} -\E_P h_{f})^2 &\leq \frac{1}{2} \int \eins_{F_r}(x)|f(x)-f^{\ast}_{L,P}(x)| \dxy{P_X}{x}\\
											&\leq \frac{c_{LC} }{2r^{\g}} \int \eins_{F_r}(x)|2\n(x)-1| |f(x)-f^{\ast}_{L,P}(x)| \dxy{P_X}{x}\\
											&= \frac{c_{LC}}{r^{\g}} \int_{(X_{-1} \triangle \lbrace f < 0 \rbrace)\cap F_r}  |2\n(x)-1| \dxy{P_X}{x}\\
											&= \frac{c_{LC}}{r^{\g}} (\RP{L_{F_r}}{f}-\RPB{L_{F_r}})\\
											&= \frac{c_{LC}}{r^{\g}}  \E_P h_{f}.\tag*{\qedhere}
\end{align*}
\end{proofof}

\begin{proofof}{Theorem \ref{Hist_theorem_oracle_inequality}}
We define the set of cubes $N_r$ and $F_r$ as in (\ref{def_A}), (\ref{def_B}) for the choice of 
\begin{align}\label{def_r}
r:=\tilde{c}_{\a,\g,d}\left( \frac{\t}{s^dn} \right)^{\frac{1-\th}{1+\g(2-\th)}},
\end{align}
where
\begin{align}\label{def_th}
\th:=\frac{\a}{\a+\g}.
\end{align} 
With (\ref{def_r}) we find that $s \leq r$. To see the latter, we remark that
\begin{align*}
s \leq \tilde{c}_{\a,\g,d}\left( \frac{\t}{s^dn} \right)^{\frac{1-\th}{1+\g(2-\th)}} & \Longleftrightarrow  s^{\frac{1+\g(2-\th)+d(1-\th)}{1+\g(2-\th)}} \leq \tilde{c}_{\a,\g,d}\left( \frac{\t}{n} \right)^{\frac{1-\th}{1+\g(2-\th)}}\\
&\Longleftrightarrow  s \leq \left(\tilde{c}_{\a,\g,d}\left( \frac{\t}{n} \right)^{\frac{1-\th}{1+\g(2-\th)}}\right) ^{\frac{1+\g(2-\th)}{1+\g(2-\th)+d(1-\th)}}
\end{align*}
and conclude by replacing $\th$ by (\ref{def_th}) that $s \leq r$ holds if
\begin{align*}
s \leq \tilde{c}_{\a,\g,d}^{\frac{(1+\g)(\a+\g)+\g^2}{(1+\g)(\a+\g)+\g^2+d\g}}\left(\frac{\t}{n} \right)^{\frac{\g}{(1+\g)(\a+\g)+\g^2+d\g}},
\end{align*} 
which equals (\ref{s_leq}). Hence, we are able to split the excess risk $\RP{L}{h_{D,s}}-\RPB{L}$ according to Lemma \ref{lemma_risksplit} by
\begin{align}\label{risksplit}
\begin{split}
&\RP{L}{h_{D,s}}-\RPB{L}\\ 
&\leq \left(\RP{L_{N_r}}{h_{D,s}} -\RPB{L_{N_r}}\right)+ \left( \RP{L_{F_r}}{h_{D,s}}-\RPB{L_{F_r}} \right).
\end{split}
\end{align}
 The rest of the proof is structured in three parts, where we establish error bounds on $N_r$ and $F_r$ in the first two parts and combine the results obtained in the third and last part of the proof. In the following we write $N:=N_r$ and $F:=F_r$ and keep in mind, that these sets depend on a parameter $r$. Furthermore, we write $h_{D}:=h_{D,s}$.

\input{hist_oracle_part_B}

\input{hist_oracle_part_A}

\textit{Part 3:} 
In the last part we combine the results obtained in Part 1, the oracle inequality on $F$ and Part 2, the oracle inequality on $N$. That means, with the separation in (\ref{risksplit}) we obtain with (\ref{ineq_oracle_B}) and (\ref{ineq_oracle_A_case1}) for the oracle inequality on $X$ that
\begin{align}\label{ineq_oracle_all_case1}
\begin{split}
&\RP{L}{h_{D,s}}-\RPB{L}\\
&\leq \left( \RP{L_N}{h_{D,s}} -\RPB{L_N}\right)+ \left( \RP{L_F}{h_{D,s}}-\RPB{L_F} \right)\\
&\leq  6\left(c_{\text{MNE}}s \right)^{\b}+ 4 \left( \frac{8V(c_4rs^{-d} +\t)}{n}\right)^{\frac{1}{2-\th}} +\frac{32 c_1 ( 8^{d+1}s^{-d}+\t)}{r^{\g}n}
\end{split}
\end{align}
holds with probability $P^n \geq 1-2e^{-\t}$. Since $s\in (0,1]$ and $r \geq s$, we find that $rs^{-d} \geq 1$. Together with the fact $s^{-d}, \tau \geq 1$ and $c_4 \geq 1$ it follows that 
\begin{align*}
&\RP{L}{h_{D,s}}-\RPB{L}\\
 &\leq 6\left(c_{\text{MNE}}s \right)^{\b}+ 4 \left( \frac{8V(c_4rs^{-d} +\t)}{n}\right)^{\frac{1}{2-\th}} +\frac{32 c_1 ( 8^{d+1}s^{-d}+\t)}{r^{\g}n}\\
&\leq 6\left(c_{\text{MNE}}s \right)^{\b}+ 4 \left( \frac{8V (c_4\t rs^{-d} +c_4\t rs^{-d})}{n}\right)^{\frac{1}{2-\th}}+\frac{32 c_1 ( 8^{d+1}\t s^{-d}+\t s^{-d})}{r^{\g}n}\\
&\leq 6\left(c_{\text{MNE}}s \right)^{\b}+ 4 \left( \frac{c_5 \t r s^{-d}}{n}\right)^{\frac{1}{2-\th}}+\frac{c_6 \t s^{-d}}{r^{\g} n}\\
&\leq 6\left(c_{\text{MNE}}s \right)^{\b}+ r^{\frac{1}{2-\th}} 4\left( \frac{ c_5\t }{s^{d}n}\right)^{\frac{1}{2-\th}}+\frac{c_6\t }{r^{\g} s^{d}n},
\end{align*}
where $c_5:=32 V \max \lbrace 12\mathcal{H}^{d-1}(\lbrace \n = 1/2 \rbrace),1\rbrace$ and $c_6:=64\cdot 8^{d+1}\max\lbrace c_{LC},2^{\g}\rbrace$. Thus, inserting $r$, defined in (\ref{def_r}), with the choice of $\tilde{c}_{\a,\g,d}:=\\ \left(\frac{(\g(2-\th)c_6)^{2-\th}}{4^{2-\th}  c_5}\right)^{\frac{1}{1+\g(2-\th)}}$ minimizes the right-hand side and yields 
\begin{align*}
\begin{split}
&\RP{L}{h_{D,s}}-\RPB{L}\\
&\leq 6\left(c_{\text{MNE}}s \right)^{\b}+ r^{\frac{1}{2-\th}} 4\left( \frac{ c_5\t }{s^{d}n}\right)^{\frac{1}{2-\th}}+\frac{c_6\t }{r^{\g} s^{d}n}\\
&= 6\left(c_{\text{MNE}}s \right)^{\b}+ 4(\tilde{c}_{\a,\g,d}c_5)^{\frac{1}{2-\th}}\left( \frac{\t}{s^dn} \right)^{\frac{2-\th+\g(2-\th)}{(1+\g(2-\th))(2-\th)}}+\frac{c_6}{\tilde{c}_{\a,\g,d}^{\g}}\left(\frac{\t}{s^dn} \right)^{\frac{1+\g}{1+\g(2-\th)}}\\
&=6\left(c_{\text{MNE}}s \right)^{\b}+ \left(\frac{\tilde{c}_{\a,\g,d}^\frac{1+\g(2-\th)}{2-\th}4c_5^\frac{1}{2-\th}+c_6}{\tilde{c}_{\a,\g,d}^{\g}} \right)\left(\frac{\t}{s^dn} \right)^{\frac{1+\g}{1+\g(2-\th)}}\\
&=6\left(c_{\text{MNE}}s \right)^{\b}+ \left(\frac{\g(2-\th)c_6+c_6}{\tilde{c}_{\a,\g,d}^{\g}} \right)\left(\frac{\t}{s^dn} \right)^{\frac{1+\g}{1+\g(2-\th)}}\\
&\leq 6\left(c_{\text{MNE}}s \right)^{\b}+ \left(\frac{2c_6 \max\lbrace\g(2-\th),1\rbrace}{\tilde{c}_{\a,\g,d}^{\g}} \right)\left(\frac{\t}{s^dn} \right)^{\frac{1+\g}{1+\g(2-\th)}}
\end{split}
\end{align*}
and we find again by inserting $\th$ that
\begin{align}\label{oracle_final}
\RP{L}{h_{D,s}}-\RPB{L}&\leq 6\left(c_{\text{MNE}}s \right)^{\b}+ c_{\a,\g,d} \left( \frac{\t}{s^dn} \right)^{\frac{ (1+\g)(\a+\g)}{(1+\g)(\a+\g)+\g^2}}
\end{align}
holds with probability $P^n \geq 1-2e^{-\t}$, where $c_{\a,\g,d}:=\frac{2c_6 \max\lbrace\g(2-\th),1\rbrace}{\tilde{c}_{\a,\g,d}^{\g}}= \frac{2c_6 \max\left\lbrace \frac{\g(\a+2\g)}{\a+\g},1\right\rbrace}{\tilde{c}_{\a,\g,d}^{\g}}$.
\end{proofof}

\begin{proofof}{Theorem \ref{Hist_theorem_learning_rate}}
We begin by proving that the chosen sequence $s_n$ satisfies assumptions (\ref{s_leq}) and (\ref{s_geq}). To this end, we define $n_{\t,\a,\b,\g,d}:=\left( \frac{\tilde{c}_{\a,\b,\g,\t,d}}{c_1} \right)^{\frac{1}{\z_1}}$ with $c_1:=\tilde{c}_{\a,\g,d}^{\frac{\k+\g^2}{\k+\g^2+d\g}} \t^{\frac{\g}{\k+\g^2+d\g}}$, where $\tilde{c}_{\a,\g,d}
$  is the constant from Theorem \ref{Hist_theorem_oracle_inequality}, and $\z_1:=\frac{\k(\k+\g^2+d\g)-\g(\b(\k+\g^ 2)+d\k)}{(\b(\k+\g^ 2)+d\k)(\k+\g^2+d\g)}$. We remark that $\z_1 \geq 0$ since we find by  $\b \leq \g^{-1}(1+\g)(\a+\g)$ that 
\begin{align*}
\k(\k+\g^2+d\g)-\g(\b(\k+\g^ 2)+d\k)&={}\k^2+k\g^2-\g\b\k-\b\g^3\\
                                    &\Geq \k^2+\k\g^2-\k^2-\k\g^2\\
                                    &={}0.
\end{align*} 
Then, for $n \geq n_{\t,\a,\b,\g,d}$ a simple calculation shows that the latter is equivalent to
\begin{align*}
c_1 n^{-\frac{\g}{\k+\g^2+d\g}} \geq  \tilde{c}_{\a,\b,\g,\t,d} n^{-\frac{\k}{\b(\k+\g^2)+d\k}},
\end{align*}
which equals assumption (\ref{s_leq}) with $s_n:= \tilde{c}_{\a,\b,\g,\t,d} n^{-\frac{\k}{\b(\k+\g^2)+d\k}}$. To see that assumption (\ref{s_geq}) is satisfied we define $ \tilde{n}_{\t,\a,\b,\g,d}:=\left( \frac{c_2}{\tilde{c}_{\a,\b,\g,\t,d}}\right)^{\frac{1}{\z_2}}$ with 
$c_2:=\t^{\frac{1}{d}} \left( \frac{ \tilde{c}_{\a,\g,d}}{\min\lbrace \frac{\d^{\ast}}{3},1\rbrace} \right)^{\frac{\k+\g^2}{d\g}}$, where $\tilde{c}_{\a,\g,d}$ is the constant from Theorem \ref{Hist_theorem_oracle_inequality}, $\d^{\ast}$ the one from Lemma \ref{lemma_hausdorff} and where $\z_2:=\frac{\b(\k+\g^2)}{d(\b(\k+\g^2)+d\k)}$.
Then, a simple transformation shows again that for all $n \geq \tilde{n}_{\t,\a,\b,\g,d}$ we find
\begin{align*}
\tilde{c}_{\a,\b,\g,\t,d} n^{-\frac{\k}{(\b(\k+\g^2)+d\k)}} \geq c_2n^{-\frac{1}{d}},
\end{align*}
which equals assumption (\ref{s_geq}) with $s_n:=\tilde{c}_{\a,\b,\g,\t,d} n^{-\frac{\k}{(\b(\k+\g^2)+d\k)}}$.

Finally, we obtain for all $n \geq n_0:=\lceil\max \lbrace n_{\t,\a,\b,\g,d},\tilde{n}_{\t,\a,\b,\g,d}\rbrace \rceil$ by inserting our chosen sequence $s_n$, satisfying (\ref{s_leq}) and (\ref{s_geq}), in (\ref{oracle_theorem}) that
\begin{align*}
&\RP{L}{h_{D,s_n}}-\RPB{L}\\
&\leq 6(c_{\text{MNE}}s_n)^{\b}+ c_{\a,\g,d} \left( \frac{\t}{s_n^dn} \right)^{\frac{ \k}{\k+\g^2}}\\
&= 6c_{\text{MNE}}^{\b} \tilde{c}_{\a,\b,\g,\t,d}^{\b}n^{-\frac{\b\k}{\b(\k+\g^2)+d\k}}
+c_{\a,\g,d}\t^{\frac{\k}{\k+\g^2}} \tilde{c}_{\a,\b,\g,\t,d}^{-\frac{d\k}{\k+\g^2}}n^{-\frac{\b\k}{\b(\k+\g^2)+d\k}}
\\
&=\left(\frac{6c_{\text{MNE}}^{\b} \tilde{c}_{\a,\b,\g,\t,d}^{\frac{\b(\k+\g^2)+d\k}{\k+\g^2}}+ c_{\a,\g,d}\t^{\frac{\k}{\k+\g^2}}  }{\tilde{c}_{\a,\b,\g,\t,d}^{\frac{d\k}{\k+\g^2}}}\right) n^{-\frac{\b\k}{\b(\k+\g^ 2)+d\k}}\\
&=  \left(\frac{ \frac{d\k}{\b(\k+\g^2)}c_{\a,\g,d}\t^{\frac{\k}{\k+\g^2}} + c_{\a,\g,d}\t^{\frac{\k}{\k+\g^2}} }{\tilde{c}_{\a,\b,\g,\t,d}^{\frac{d\k}{\k+\g^2}}}\right)  n^{-\frac{\b\k}{\b(\k+\g^2)+d\k}}\\
&\leq \left( \frac{2\max\left\lbrace \frac{d\k}{\b(\k+\g^2)},1\right\rbrace c_{\a,\g,\d}\t^{\frac{\k}{\k+\g^2}}}{\tilde{c}_{\a,\b,\g,\t,d}^{  \frac{d\k}{\k+\g^2}}}\right) n^{-\frac{\b\k}{\b(\k+\g^2)+d\k}}\\
&=c_{\a,\b,\g,\t,d}n^{-\frac{\b\k}{\b(\k+\g^2)+d\k}}
\end{align*}
holds with probability $P^n \geq 1-2e^{-\t}$, where $c_{\a,\b,\g,\t,d}:=\\ 2\max\left\lbrace \frac{d\k}{\b(\k+\g^2)},1\right\rbrace c_{\a,\g,\d}\t^{\frac{\k}{\k+\g^2}} \cdot \tilde{c}_{\a,\b,\g,\t,d}^{-\frac{d\k}{\k+\g^2}}$.\qedhere
\end{proofof}

\begin{proofof}{Theorem \ref{theorem_learning_rate_TV}}
Let $s_n^{\ast}$ behave as $s_n$ in Theorem \ref{Hist_theorem_learning_rate}, that is
$s_n^{\ast}\sim n^{-\frac{\k}{\b(\k+\g^2)+d\k}}$.
We assume that $S_n:=\lbrace s_1^{(n)}, \ldots, s_l^{(n)} \rbrace$ and $s_{i-1}^{(n)}<s_i^{(n)}$ for $i\in \lbrace 2, \ldots ,l \rbrace$. Since $S_n$ is a $n^{-1/d}$-net we have 
\begin{align}\label{s_i}
s_i^{(n)}-s_{i-1}^{(n)}\leq 2n^{-1/d}.
\end{align} 
Furthermore, there exists indices $i \in \lbrace 1, \ldots ,l \rbrace$ such that $s_{i-1}^{(n)} \leq s_n^{\ast} \leq s_i^{(n)}$.
An analogous calculation as at the beginning of the proof of Theorem \ref{Hist_theorem_learning_rate} shows then that $s_{i-1}^{(n)}$ and $s_i^{(n)}$ satisfy assumption (\ref{s_leq}) and (\ref{s_geq}) for sufficiently large $n$.
%
Hence, we find for $s \in \lbrace s_{i-1}^{(n)},s_i^{(n)} \rbrace $ with Theorem \ref{Hist_theorem_oracle_inequality}
that 
\begin{align}\label{ineq_oracle1_TV}
\RP{L}{h_{D_1,s}}-\RPB{L} \leq  6\left(c_{\text{MNE}}s \right)^{\b}+ c_{\a,\g,d} \left( \frac{\t}{s^dk} \right)^{\frac{\k}{\k+\g^2}}.
\end{align}
holds with $P^k \geq 1-4e^{-\t}$.
Since $h_{D_1,s^{\ast}_{D_2}}$ is an ERM we find with \cite[Theorem~7.2, Theorem~8.24 and Exercise~8.5]{StCh08} and $\t_n:=\t+\log(1+2|S_n|)$ that
\begin{align}\label{ineq_erm_TV}
\begin{split}
&\RP{L}{h_{D_1,s^{\ast}_{D_2}}}-\RPB{L}\\
&< 6 \inf_{s \in S_n} \left(\RP{L}{h_{D_1,s}}-\RPB{L}\right)+4\left(\frac{8c_{\a,\g}\t_n}{n-k} \right)^{\frac{\a+\g}{\a+2\g}}\\
&\leq \inf_{s \in  \lbrace s_{i-1}^{(n)},s_i^{(n)} \rbrace} \left(\RP{L}{h_{D_1,s}}-\RPB{L}\right)+4\left(\frac{8c_{\a,\g}\t_n}{n-k} \right)^{\frac{\a+\g}{\a+2\g}}
\end{split}
\end{align}
holds with $P^{n-k}\geq 1-e^{-\t}$. Combining (\ref{ineq_oracle1_TV}) and (\ref{ineq_erm_TV}) we obtain with $k\geq n/2$ and $n-k = n/2+n/2-k \geq n/4$  that
\begin{align*}
&\RP{L}{h_{D_1,s^{\ast}_{D_2}}}-\RPB{L}\\
&\leq 6 \inf_{s \in \lbrace s_{i-1}^{(n)},s_i^{(n)} \rbrace} \left(6\left(c_{\text{MNE}}s \right)^{\b}+ c_{\a,\g,d} \left( \frac{2\t}{s^dn} \right)^{\frac{\k}{\k+\g^2}} \right)+4\left(\frac{32c_{\a,\g}\t_n}{n} \right)^{\frac{\a+\g}{\a+2\g}}\\
&\leq c_1 \left( \inf_{s \in \lbrace s_{i-1}^{(n)},s_i^{(n)} \rbrace} \left(s^{\b}+ \left( \frac{2\t}{s^dn} \right)^{\frac{\k}{\k+\g^2}}\right)+\left(\frac{\t_n}{n} \right)^{\frac{\a+\g}{\a+2\g}} \right)
\end{align*}
holds with $P^n \geq 1-(1+4)e^{-\t}$. With Lemma \ref{Appendix3} we find that
\begin{align}\label{ineq_cons_TV}
\begin{split}
&\RP{L}{h_{D_1,s^{\ast}_{D_2}}}-\RPB{L}\\
&\leq c_1 \left( \inf_{s \in \lbrace s_{i-1}^{(n)},s_i^{(n)} \rbrace} \left(s^{\b}+ \left( \frac{2\t}{s^dn} \right)^{\frac{\k}{\k+\g^2}}\right)+\left(\frac{\t_n}{n} \right)^{\frac{\a+\g}{\a+2\g}} \right)\\
&\leq  c_2 \left(  \left(s_n^{\ast}\right)^{\b}+ \left( \frac{2\t}{\left(s_n^{\ast}\right)^dn} \right)^{\frac{\k}{\k+\g^2}}+\left(\frac{\t_n}{n} \right)^{\frac{\a+\g}{\a+2\g}}  +n^{-\frac{\b}{d}}  \right)\\
&\leq c_2 \left( \left(s_n^{\ast}\right)^{\b}+ \left( \frac{2\t_n}{\left(s_n^{\ast}\right)^dn} \right)^{\frac{\k}{\k+\g^2}}+\left(\frac{2\t_n}{n} \right)^{\frac{\a+\g}{\a+2\g}} +n^{-\frac{\b}{d}}  \right)
\end{split}
\end{align}
holds with $P^n \geq 1-5e^{-\t}$. 
Next, it is easy to verify with $\b \leq \g^{-1}\k$ that
\begin{align*}
\left( \frac{2\t_n}{\tilde{s}_n^d n} \right)^{\frac{\k}{\k+\g^2}}  \geq \left(\frac{2\t_n}{n} \right)^{\frac{\a+\g}{\a+2\g}}
\end{align*}
such that we can omit the latter right-hand side term in (\ref{ineq_cons_TV}). Hence, we have that 
\begin{align*}
&\RP{L}{h_{D_1,s^{\ast}_{D_2}}}-\RPB{L}\\
&\leq c_2 \left(  \left(s_n^{\ast}\right)^{\b}+ \left( \frac{2\t_n}{\left(s_n^{\ast}\right)^dn} \right)^{\frac{ (1+\g)(\a+\g)}{(1+\g)(\a+\g)+\g^2}}+\left(\frac{2\t_n}{n} \right)^{\frac{\a+\g}{\a+2\g}} +n^{-\frac{\b}{d}} \right)\\
&\leq c_3  \left(\left(s_n^{\ast}\right)^{\b}+ \left( \frac{2\t_n}{\left(s_n^{\ast}\right)^dn} \right)^{\frac{ (1+\g)(\a+\g)}{(1+\g)(\a+\g)+\g^2}} +n^{-\frac{\b}{d}}\right)\\
&\leq c_3 \left( n^{-\frac{\b\k}{\b(\k+\g^2)+d\k}}+n^{-\frac{\b}{d}} \right)\\
&\leq c_4 \cdot n^{-\frac{\b\k}{\b(\k+\g^2)+d\k}} 
\end{align*}
holds with $P^n \geq 1-5e^{-\t}$, where in the last step we used that 
\begin{align*}
\frac{\b\k}{\b(\k+\g^2)+d\k} \leq \frac{\b\k}{d\k} = \frac{\b}{d}.
\end{align*}  
\end{proofof}

%% file: hist_oracle_part_B.tex
\textit{Part 1:} In the first part we establish an oracle inequality for $\RP{L_F}{h_{D,s}}-\RPB{L_F}$. Therefore we define $h^F_f:=L_F \circ f-L_F \circ f^{\ast}_{L_F,P}$ and find that 
\begin{align*}
\|h^F_{f}\|_{\infty}=\|L_F \circ f-L_F \circ f^{\ast}_{L_F,P}\|_{\infty} \leq 1
\end{align*}
for all $f \in \F$. Furthermore, with Lemma \ref{Lemma_var_bound} we obtain
\begin{align}\label{calc_VB_B2}
\begin{split}
\E_P (h^F_{f})^2 &\leq \frac{c_{\text{LC}}}{r^{\g}}\E_P h^F_{f} \leq \frac{c_1}{r^{\g}}\E_P h^F_{f},
\end{split}
\end{align}
where $c_1:=\max \lbrace c_{\text{LC}},2^{\g} \rbrace$. We observe that $r^{\g} \leq c_1$, since with assumption (\ref{s_geq}), where we rewrite the exponent by $\frac{(1+\g)(\a+\g)+\g^2}{\g}=\frac{1+\g(2-\th)}{1-\th}$, we find
\begin{align*}
r&=\tilde{c}_{\a,\g,d}\left( \frac{\t}{s^dn} \right)^{\frac{1-\th}{1+\g(2-\th)}}\\
 &\leq \tilde{c}_{\a,\g,d}\left(  \left( \frac{\min\lbrace \frac{\d^{\ast}}{3},1\rbrace}{\tilde{c}_{\a,\g,d}} \right)^{\frac{1+\g(2-\th)}{1-\th}}\right)^{\frac{1-\th}{1+\g(2-\th)}}\\
 &= \min\left\lbrace \frac{\d^{\ast}}{3},1 \right\rbrace\\
 &\leq 1
\end{align*} 
and therefore $r^{\g} \leq 2^{\g} \leq c_1$. As we conclude from Lemma \ref{lemma_erm_subset} that $h_D$ is an empirical risk minimizer over $\F$ for the loss $L_F$, we are able to use \cite[Theorem~7.2]{StCh08}, an improved oracle inequality for ERM. We obtain for all fixed $\t \geq 1$ and $n\geq 1$ that
\begin{align*}
\RP{L_F}{h_{D}}-\RPB{L_F}<6 (\RPxB{L_F}{\F}-\RPB{L_F})+  \frac{32 c_1 (\text{log}(|\F|+1)+\t)}{r^{\g}n}
\end{align*}
holds with probability $P^n \geq 1-e^{-\t}$, where $\RPxB{L_F}{\F}:= \inf_{f\in \F}\RP{L_F}{f}$. Next, we refine the right-hand side of this oracle inequality. Obviously we have $|\F| \leq 2^{|J|}$. We bound the the cardinality $|J|$ by using a volume comparison argument. To this end, we define the set $\tilde{J} := \lbrace \,j \geq 1 \,|\,A_j \cap 2\left[-1,1 \right]^d  \neq \emptyset \,\rbrace$ and observe that \mbox{$\bigcup_{j\in J} A_j \subset \bigcup_{j\in \tilde{J}} A_j \subset 4B_{\ell^d_{\infty}}$}. 
Then, 
\begin{align*}
|J| s^d = \lb^d \left( \bigcup_{j\in J} A_j \right) \leq \lb^d \left( \bigcup_{j\in \tilde{J}} A_j \right)   \leq \lb^d  \left( 4B_{\ell^d_{\infty}} \right) = 8^d,
\end{align*}
such that we deduce with $|J| \leq 8^ds^{-d}$ that
\begin{align*}
\text{log}(|\F|+1)  &\leq \text{log}(2^{8^ds^{-d}}+1)\\
					&\leq \text{log}(2\cdot 2^{8^ds^{-d}})\\
					&= \text{log}(2^{8^ds^{-d}+1})\\
					&= (8^ds^{-d}+1)\text{log}(2)\\
					&\leq 8^ds^{-d}+1\\
					&\leq 8^{d+1}s^{-d}.
\end{align*}
Thus,
\begin{align}\label{oracle_B_withoutAE}
\RP{L_F}{h_{D}}-\RPB{L_F}<6 (\RPxB{L_F}{\F}-\RPB{L_F})+ \frac{32 c_1(8^{d+1}s^{-d}+\t)}{r^{\g} n}
\end{align}
holds with probability $P^n \geq 1-e^{-\t}$.

Finally, we have to bound the \textit{approximation error} $\RPxB{L_F}{\F}-\RPB{L_F}=\inf_{f\in \F}\RP{L_F}{f}-\RPB{L_F}$. We find with $h_{P,s} \in \F$ and Lemma \ref{Appendix1} that 
\begin{align}\label{approx_error_B}
\begin{split}
\RPxB{L_F}{\F}-\RPB{L_F}&\leq \RP{L_F}{h_{P,s}}-\RPB{L_F}\\
						&= \int_{(X_1 \triangle \lbrace h_{P,s}\geq 0 \rbrace)\cap F}|2\n-1|\,dP_X\\
&= \sum_{j \in J_F^r} \int_{(X_1 \triangle \lbrace h_{P,s}\geq 0 \rbrace)\cap A_j}|2\n-1|\,dP_X\\
&= 0,
\end{split}
\end{align}
since $P_X((X_1 \triangle \lbrace h_{P,s}\geq 0 \rbrace)\cap A_j)=0$ for each $j \in J_F^r$. To see the latter, we first remark that the latter set contains those $x \in A_j$ for that either $h_{P,s}(x) \geq 0$ and $\n(x) \leq 1/2$ or $h_{P,s}(x) < 0$ and $\n(x) > 1/2$. Since we have $A_j \subset X_{-1} \cup X_1$ we can ignore the case $\n(x)=1/2$. Furthermore, we know by Lemma \ref{lemma_sets} i) that either $A_j \cap X_{-1} = \emptyset$ or
$A_j \cap X_{1} = \emptyset$. Let us first consider the case $A_j \cap X_{-1} = \emptyset$ and thus $A_j \subset X_1$. According to the definition of the histogram rule, cf. (\ref{def_hist}), we find for all $x \in A_j$ that $h_{P,s}(x)=1$, since
\begin{align*}
&f_{P,s}(x)\\
&= P(A_j(x)\times \lbrace 1 \rbrace)-P(A_j(x) \times \lbrace -1 \rbrace )\\
&=\int_{A_j}\int_{Y} \eins_{A_j \times \lbrace 1 \rbrace}(x,y) P(dy|x)\dxy{P_X}{x}\\
&\qquad-\int_{A_j}\int_{Y} \eins_{A_j \times \lbrace -1 \rbrace}(x,y) P(dy|x)\dxy{P_X}{x}\\
&=\int_{A_j}\eins_{A_j \times \lbrace 1 \rbrace}(x,1) \n(x)\dxy{P_X}{x}-\int_{A_j}\eins_{A_j \times \lbrace -1 \rbrace}(x,-1)(1-\n(x))\dxy{P_X}{x} \\
&= \int_{A_j}2\eta(x)-1 \dxy{P_X}{x}\\
&\geq 0.
\end{align*}
Obviously we have $\n(x) \geq 1/2$ and $h_{P,s}(x)=1$ for all $x \in A_j$. Analogously we can show for cells with $A_j \cap X_{1} = \emptyset$ for $j \in J_F^r$ that $\n(x) \leq 1/2$ and $h_{P,s}(x)=-1$ for all $x \in A_j$. Hence, $P_X((X_1 \triangle \lbrace h_{P,s}\geq 0 \rbrace)\cap A_j)=0$ for all $j \in J_F^r$ and the approximation error vanishes on the set $F$.

Altogether, for the oracle inequality on $F$ we obtain with (\ref{oracle_B_withoutAE}) and (\ref{approx_error_B}) that 
\begin{align}\label{ineq_oracle_B}
\RP{L_F}{h_{D}}-\RPB{L_F}< \frac{32 c_1( 8^{d+1}s^{-d}+\t)}{r^{\g} n}
\end{align}
holds with probability $P^n \geq 1-e^{-\t}$.

%% file: hist_oracle_part_A.tex
\emph{Part 2:} In the second part we establish an oracle inequality for $\RP{L_N}{h_{D}}-\RPB{L_N}$, again by using \cite[Theorem~7.2]{StCh08}. Analogously to Part 1 we define $h^N_f:=L_N \circ f-L_N\circ f^{\ast}_{L_N,P}$ for $f \in \F$ and find $\|h^N_{f}\|_{\infty} \leq 1$. Since $(h^N_{f_0})^2 =\eins_{N}\frac{|f-f^{\ast}_{L,P}|}{2}= \eins_{(X_{-1} \triangle \lbrace f < 0 \rbrace)\cap N}$ we find with [Appendix, Lemma \ref{Appendix1}] that
\begin{align}\label{calc_VB_A_part1_1}
\begin{split}
&\E_P (h^N_{f_0})^2\\
 &={} \frac{1}{2} \int_N|f_0(x)-f^{\ast}_{L_N,P}(x)| \dxy{P_X}{x}\\
											&={}\frac{1}{2} \int_{N \cap \lbrace |2\n-1| \geq t \rbrace}|f_0(x)-f^{\ast}_{L_N,P}(x)| \dxy{P_X}{x}\\
											&\qquad+ \frac{1}{2} \int_{N \cap \lbrace |2\n-1| < t \rbrace}|f_0(x)-f^{\ast}_{L_N,P}(x)| \dxy{P_X}{x}\\
											&\Leq \frac{1}{2t} \int_{N \cap \lbrace |2\n-1| \geq t \rbrace}|2\n (x)-1||f_0(x)-f^{\ast}_{L_N,P}(x)| \dxy{P_X}{x}\\
											&\qquad+ P_X(\lbrace{ x \in N: |2\n(x)-1|<t \rbrace})\\
											&\Leq \frac{1}{2t} \int_N |2\n (x)-1||f_0(x)-f^{\ast}_{L_N,P}(x)| \dxy{P_X}{x}\\
											&\qquad+ P_X(\lbrace{ x \in N: |2\n(x)-1|<t \rbrace})\\
											&\Leq t^{-1}\E_P h^N_{f_0} + \min \lbrace P_X(N),P_X(\lbrace x \in X:|2\n(x)-1|<t \rbrace) \rbrace
											\end{split}
\end{align}
for all $t>0$. We turn our attention to the minimum and note that by the definition of $N$ we have 
\begin{align}\label{calc_P(A)}
P_X(N)\leq P_X(\lbrace \D_{\n}(x) \leq 3r \rbrace).
\end{align} 
For $x \in X$ with $|2\n(x) -1|<t$ by the definition of the lower control we conclude from
\begin{align*}
\frac{\D_{\n}^{\g}(x)}{c_{\text{LC}}} \leq |2\n(x) -1|<t.
\end{align*} 
that
\begin{align*}
\D_{\n}(x) \leq  (c_{\text{LC}}t)^{\frac{1}{\g}}
\end{align*}
and consequently
\begin{align}\label{calc_2eta-1}
\lbrace x \in X:|2\n(x)-1|<t \rbrace \subset \lbrace x \in X:\D_{\n}(x) \leq (c_{\text{LC}}t)^{\frac{1}{\g}} \rbrace.
\end{align}
Then, we find by (\ref{calc_P(A)}), (\ref{calc_2eta-1}) and by the definition of the margin exponent that
\begin{align}\label{calc_min}
\begin{split}
 &\min \lbrace P_X(N),P_X(\lbrace x \in X:|2\n(x)-1|<t \rbrace) \rbrace\\
											&\leq \min \lbrace P_X(\lbrace \D_{\n}(x) \leq 3r \rbrace),P_X(\lbrace x \in X:\D_{\n}(x) \leq (c_{\text{LC}}t)^{\frac{1}{\g}}  \rbrace) \rbrace\\
											&\leq\min \lbrace  (c_{\text{ME}}3r)^{\a},c_{\text{ME}}^{\a}(c_{\text{LC}}t)^{\frac{\a}{\g}}\rbrace.
\end{split}
\end{align}
Combining (\ref{calc_min}) with (\ref{calc_VB_A_part1_1}) we obtain
\begin{align}\label{ineq_right_hand_side}
\begin{split}
\E_P (h^N_{f_0} -\E_P h^N_{f_0})^2 &\leq t^{-1}\E_P h^N_{f_0} + \min \lbrace  (c_{\text{ME}}3r)^{\a},c_{\text{ME}}^{\a}(c_{\text{LC}}t)^{\frac{\a}{\g}}\rbrace\\
                                   &\leq t^{-1}\E_P h^N_{f_0} + c_{\text{ME}}^{\a}(c_{\text{LC}}t)^{\frac{\a}{\g}}.
\end{split}
\end{align}
Minimizing the right-hand side of (\ref{ineq_right_hand_side}) yields
\begin{align*}
\min_{t>0} \left( t^{-1}\E_P h^N_{f_0} +  c_{\text{ME}}^{\a}(c_{\text{LC}}t)^{\frac{\a}{\g}} \right)= c_2 \left(\E_P h^N_{f_0}\right)^{\frac{\a}{\a+\g}},
\end{align*}
where $c_2:=\frac{\a+\g}{\g}c_{\text{ME}}^{\frac{\a\g}{\a+\g}}\left( \frac{\g c_{\text{LC}}}{\a}\right)^{\frac{\a}{\a+\g}}$, such that with 
\begin{align}
V &:=\max\lbrace 1, c_2 \rbrace \label{def_V}
\end{align}
and (\ref{def_th}) we have 
\begin{align}\label{calc_var_bound}
\begin{split}
\E_P (h^N_{f_0})^2 \leq t^{-1}\E_P h^N_{f_0} + c_{\text{ME}}(c_{\g}t^{\frac{1}{\g}})^{\a}=c_2\left(\E_P h^N_{f_0}\right)^{\frac{\a}{\a+\g}}\leq V \left(\E_P h^N_{f_0}\right)^{\th}.
\end{split}
\end{align}
Note, that the definition of $V$ yields $V^{\frac{1}{2-\th}} \geq 1$. Since $h_D$ is an ERM over $\F$ for the loss $L_N$ due to Lemma \ref{lemma_erm_subset}, by using \cite[Theorem~7.2]{StCh08} we obtain for fixed $\t \geq 1$ and $n\geq 1$ that 
\begin{align}\label{ineq_bern_F}
\begin{split}
&\RP{L_N}{h_{D}}-\RPB{L_N}\\
&<6 (\RPxB{L_N}{\F}-\RPB{L_N})+ 4 \left( \frac{8V(\text{log}(|\F|+1) +\t)}{n}\right)^{\frac{1}{2-\th}}
\end{split}
\end{align}
holds with probability $P^n \geq 1-e^{-\t}$. In order to refine the right-hand side in (\ref{ineq_bern_F}), we establish a bound on the cardinality $|\F|=2^{|J_A|}$ and on the approximation error. To bound the mentioned cardinality we use the fact that $N$ is contained in a tube around the decision line, that is $\bigcup_{j\in J_N} A_j \subset \lbrace \D_{\n}(x)\leq 3r\rbrace$, see $(\ref{def_A})$. We remark that $3r \leq \d^{\ast}$ holds, where $\d^{\ast}$ is the constant from Lemma \ref{lemma_hausdorff}, since with assumption $(\ref{s_geq})$ we have
\begin{align*}
3r=3\tilde{c}_{\a,\g,d} \left( \frac{\t}{s^dn} \right)^{\frac{1-\th}{1+\g(2-\th)}}\leq 3\min\left\lbrace \frac{\d^{\ast}}{3},1\right\rbrace \leq \d^{\ast}.
\end{align*}
Then, with Lemma \ref{lemma_hausdorff} we find that
\begin{align*}
\lb^d(\lbrace \D_{\n}(x)\leq 3r \rbrace) \leq 12 \mathcal{H}^{d-1}(\lbrace \n = 1/2 \rbrace) r
\end{align*}
and we obtain
\begin{align*}
|J_A|s^d = \lb^d \left( \bigcup_{j\in J_A} A_j \right) \leq \lb^d(\lbrace \D_{\n}(x)\leq 3r \rbrace) \leq 12 \mathcal{H}^{d-1}(\lbrace \n = 1/2 \rbrace) r.
\end{align*}
This yields to
\begin{align*}
|J_A| \leq 12 \mathcal{H}^{d-1}(\lbrace \n = 1/2 \rbrace) r s^{-d} = c_3 r  s^{-d},
\end{align*}
where $c_3:=12 \mathcal{H}^{d-1}(\lbrace \n = 1/2 \rbrace)$. By $r\geq s \geq s^d$ we hence conclude that 
\begin{align}\label{calc_log_bound}
\begin{split}
\text{log}(|\F|+1) &\leq \text{log}(2^{c_3 r  s^{-d}}+1)\\
				  &\leq \text{log}(2 \mycdot 2^{c_3 r  s^{-d}})\\
				  &= \text{log}(2^{c_3 r  s^{-d}+1})\\
				  &= (c_3 r  s^{-d}+1)\text{log}(2)\\
				  &\leq c_3 r  s^{-d}+r  s^{-d}\\
				  &\leq c_4 r  s^{-d},
 \end{split}
\end{align}
where $c_4:=2\max \lbrace 12\mathcal{H}^{d-1}(\lbrace \n = 1/2 \rbrace),1\rbrace$. Thus, (\ref{ineq_bern_F}) changes to  
\begin{align}\label{ineq_oracle_case1_before_approx}
\RP{L_N}{h_{D}}-\RPB{L_N}&\leq 6 (\RPxB{L_N}{\F}-\RPB{L_N})+ 4 \left( \frac{8V(c_4rs^{-d} +\t)}{n}\right)^{\frac{1}{2-\th}}
\end{align}
with probability $P^n \geq 1-e^{-\t}$. 

Finally, we have to bound the \textit{approximation error} $\RPxB{L_N}{\F}-\RPB{L_N}$ in (\ref{ineq_oracle_case1_before_approx}). For $f_0=h_{P,s}$ we have with Lemma \ref{Appendix1} that
\begin{align*}
\RP{L_N}{h_{P,s}}-\RPB{L_N}&= \int_{(X_1 \triangle \lbrace h_{P,s}\geq 0 \rbrace)\cap N }|2\n-1|\,dP_X\\
&=\sum_{j \in J_N^r} \int_{(X_1 \triangle \lbrace h_{P,s}\geq 0 \rbrace)\cap A_j}|2\n-1|\,dP_X.
\end{align*} 
We split $J_N^r$ in indices where cells do not intersect the decision line and those which do by 
\begin{align*}
J_{N_1}^r &:= \lbrace \,j \in J_N^r\,|\, P_X(A_j \cap X_1)=0 \vee P_X(A_j \cap X_{-1}) =0 \, \rbrace\\
J_{N_2}^r &:= \lbrace \,j \in J_N^r \,|\, P_X(A_j \cap X_1)>0 \wedge P_X(A_j \cap X_{-1}) >0 \, \rbrace.
\end{align*}
such that
\begin{align*}
&\sum_{j \in J_N^r} \int_{(X_1 \triangle \lbrace h_{P,s}\geq 0 \rbrace)\cap A_j}|2\n-1|\,dP_X\\
&=\sum_{j \in J_{N_1}^r} \int_{(X_1 \triangle \lbrace h_{P,s}\geq 0 \rbrace)\cap A_j}|2\n-1|\,dP_X\\
&\qquad+\sum_{j \in J_{N_2}^r} \int_{(X_1 \triangle \lbrace h_{P,s}\geq 0 \rbrace)\cap A_j}|2\n-1|\,dP_X.
\end{align*} 

We notice that, as in the calculation of the approximation error in Part 1, the first sum vanishes since $P_X((X_1 \triangle \lbrace h_{P,s}\geq 0 \rbrace)\cap A_j)=0$ for all $j \in J_{N_1}^r$. Moreover, we remark that $J_{N_2}^r$ only contains cells of width $s$ that intersect the decision boundary. Hence, by using the margin-noise assumption we find
\begin{align}\label{calc_approximation_errror_A}
\begin{split}
\RP{L_N}{h_{P,s}}-\RPB{L_N}&= \sum_{j \in J_{N_2}^r} \int_{(X_1 \triangle \lbrace h_{P,s}\geq 0 \rbrace)\cap A_j}|2\n-1|\,dP_X\\
&\leq \int_{\lbrace \D_{\n}(x)\leq s \rbrace} |2\n-1|\,dP_X\\
&\leq \left(c_{\text{MNE}}s \right)^{\b}.
\end{split}
\end{align}
Altogether for the oracle inequality on N with (\ref{ineq_oracle_case1_before_approx}) we find that
\begin{align}\label{ineq_oracle_A_case1}
\RP{L_N}{h_{D}}-\RPB{L_N}\leq 6\left(c_{\text{MNE}}s \right)^{\b}+ 4 \left( \frac{8V(c_4rs^{-d} +\t)}{n}\right)^{\frac{1}{2-\th}}
\end{align}
holds with probability $P^n \geq 1-e^{-\t}$.

%% file: hist_appendix.tex
\begin{lemma}\label{Appendix1}
Let $Y:=\lbrace -1,1 \rbrace$ and $P$ be a probability measure on $X \times Y$. For $\n(x):=P(y=1|x), x \in X$ define the set $X_{1}:=\set{x \in X}{\n(x)>1/2}$. Let $L$ be the classification loss and consider for $A \subset X$ the loss $L_A(x,y,t):=\eins_{A}(x)L(y,t)$, where $y\in Y, t \in \mathbb{R}$. For a measurable $f \colon X \to \R$ we then have
\begin{align*}
\RP{L_A}{f}-\RPB{L_A} &= \int_{(X_1\triangle \lbrace f \geq 0 \rbrace)\cap A} |2\n(x)-1| \dxy{P_X}{x},
\end{align*}
where $\triangle$ denotes the symmetric difference.
\end{lemma}

\begin{proof}[Proof of Lemma \ref{Appendix1}]
It is well known, e.g., \cite[Example~3.8]{StCh08}, that 
\begin{align}\label{Appendix_ueberschuss}
\begin{split}
&\RP{L_A}{f}-\RPB{L_A}\\&= \int_A  |2\n(x)-1| \cdot \eins_{(-\infty,0)} ((2\n(x)-1)\text{sign}f(x)) \dxy{P_X}{x}.
\end{split}
\end{align}
Next, for $P_X$-almost all $x \in  A$ we have
\begin{align*}
\eins_{(-\infty,0]} ((2\n(x)-1)\text{sign}f(x)) = 1 \Leftrightarrow (2\n(x)-1)\text{sign}f(x) \leq 0.
\end{align*}
The latter is true if for $x \in A$ holds that $f(x) < 0$ and $\n(x)> 1/2 $ or that $f(x) \geq 0$ and $\n(x)\leq 1/2 $ or that $\n(x)=1/2$. However, for $\n(x)=1/2$ we have $|2\n(x)-1|=0$ and hence this case can be ignored. Then, the latter obviously equals the set $(X_1\triangle \lbrace f \geq 0 \rbrace)\cap A$ and we obtain in (\ref{Appendix_ueberschuss})
\begin{align*}
\RP{L_A}{f}-\RPB{L_A} &= \int_{(X_1\triangle \lbrace f \geq 0 \rbrace)\cap A } |2\n(x)-1| \dxy{P_X}{x}.
\tag*{\qedhere}
\end{align*}
\end{proof}


%

%

\begin{lemma}\label{Appendix2}
Let $X:=\left[-1,1 \right]^d$ and $P$ be a probability measure on $X \times \{-1,1\}$ with fixed version $\n \colon X \to [0,1]$ of its posterior probability. Then, if $\n$ is H\"older-continuous with exponent $\g$, we have that $\D_{\n}$ controls the noise from above with exponent $\g$, that means there exists a constant 
$c_{\text{UC}}>0$ such that
\begin{align*}
|2\n(x) -1| \leq c_{\text{UC}} \D_{\n}^{\g}(x)
\end{align*}
for $P_X$-almost all $x \in X$.
\end{lemma}

\begin{proof}[Proof of Lemma \ref{Appendix2}]
Fix w.l.o.g.\ an $x \in X_1$. Then, $\n(x)>1/2$. Since $\n$ is H\"older-continuous with exponent $\g$, there exists a constant $c>0$ such that we have 
\begin{align*}
|2\n(x)-1|=2|\n(x)-1/2|\leq 2|\n(x)-\n(x')|\leq 2c (d(x,x'))^{\g}
\end{align*}
for all $x'\in X_{-1}$ and hence
\begin{align*}
|2\n(x)-1| \leq 2c \inf_{\tilde{x}\in X_{-1}}(d(x,\tilde{x}))^{\g}=c_{\text{UC}} \D_{\n}^{\g}(x),
\end{align*}
where $c_{\text{UC}}:=2c$.
Obviously, the last inequality holds immediately for $x\in X$ with $\n(x)=1/2$.
\end{proof}

\begin{lemma}\label{Appendix3}
Let $\b,\g,\k,\t$ be as in Theorem \ref{Hist_theorem_learning_rate} and let $S_n:=\lbrace s_1^{(n)}, \ldots, s_l^{(n)} \rbrace$ and $s_{i-1}^{(n)},s_i^{(n)} \in S_n$ for $i \in \lbrace 1, \ldots ,l \rbrace$ be given as in the beginning of the proof of Theorem \ref{theorem_learning_rate_TV}. Furthermore, let $c_1,c_2>0$ be constants.  Then, we have 
\begin{align*}
&\inf_{s \in \lbrace{s_{i-1}^{(n)},s_i^{(n)} \rbrace}} \left(6\left(c_1 s \right)^{\b}+ c_2 \left( \frac{2\t}{s^dn} \right)^{\frac{\k}{\k+\g^2}}\right)\\&\leq  6\left(c_1s_n^{\ast}\right)^{\b}+ c_2 \left( \frac{2\t}{\left(s_n^{\ast}\right)^dn} \right)^{\frac{\k}{\k+\g^2}}  +6\left(\frac{2c_1}{n^{1/d}}\right)^{\b}.
\end{align*}
\end{lemma}

\begin{proof}[Proof of Lemma \ref{Appendix3}]
For $\d>0$ we fix $\tilde{s} \in (0,1]$ such that
\begin{align*}
6\left(c_1\tilde{s}\right)^{\b}+ c_2 \left( \frac{2\t}{\left(\tilde{s}\right)^dn} \right)^{\frac{\k}{\k+\g^2}}\leq  6\left(c_1s_n^{\ast}\right)^{\b}+ c_2 \left( \frac{2\t}{\left(s_n^{\ast}\right)^dn} \right)^{\frac{\k}{\k+\g^2}} + \d,
\end{align*}
where $s_n^{\ast}$ is given as in the proof of Theorem \ref{theorem_learning_rate_TV}. Then, we find that $s_{i-1}^{(n)} \leq \tilde{s} \leq s_i^{(n)}$ and  with (\ref{s_i}) that 
\begin{align*}
\tilde{s} \leq s_i^{(n)} \leq \tilde{s} +2n^{-1/d}.
\end{align*}
Hence,
\begin{align*}
& 6\left(c_1s_i^{(n)} \right)^{\b}+ c_2 \left( \frac{2\t}{\left(s_i^{(n)}\right)^dn} \right)^{\frac{\k}{\k+\g^2}}\\
&\leq 6\left(c_1(\tilde{s}+2n^{1/d}) \right)^{\b}+ c_2 \left( \frac{2\t}{\left(\tilde{s}\right)^d n} \right)^{\frac{\k}{\k+\g^2}}\\
&\leq 6\left(c_1\tilde{s}\right)^{\b}+ c_2 \left( \frac{2\t}{\left(\tilde{s}\right)^dn} \right)^{\frac{\k}{\k+\g^2}}\\
&\leq \left(c_1s_n^{\ast}\right)^{\b}+ c_2 \left( \frac{2\t}{\left(s_n^{\ast}\right)^dn} \right)^{\frac{\k}{\k+\g^2}}+\d  +6\left(\frac{2c_1}{n^{1/d}}\right)^{\b}
\end{align*}
and finally 
\begin{align*}
&\inf_{s \in \lbrace{s_{i-1}^{(n)},s_i^{(n)} \rbrace}} \left(6\left(c_1 s \right)^{\b}+ c_2 \left( \frac{2\t}{s^dn} \right)^{\frac{\k}{\k+\g^2}}\right)\\&\leq 6\left(c_1s_n^{\ast}\right)^{\b}+ c_2 \left( \frac{2\t}{\left(s_n^{\ast}\right)^dn} \right)^{\frac{\k}{\k+\g^2}}  +6\left(\frac{2c_1}{n^{1/d}}\right)^{\b}.
\end{align*}
\end{proof}